\documentclass[11pt]{article}
\usepackage{amsmath}
\usepackage{amssymb,amsbsy}
\usepackage[dvips]{graphicx}
\usepackage{caption}
\usepackage{subcaption}
\usepackage{pstricks}
\usepackage{pst-eps}
\usepackage{pst-node}
\usepackage{enumerate}
\usepackage{dsfont}
\usepackage{natbib}
\usepackage[margin=1in]{geometry}
\usepackage{tikz} \usetikzlibrary{arrows,decorations.pathmorphing,backgrounds,positioning,fit,petri,mindmap,shapes.geometric,decorations.pathreplacing}

\usepackage{pdfsync}
\usepackage{hyperref}

\allowdisplaybreaks[1]

\newtheorem{theorem}{Theorem}
\newtheorem{proposition}{Proposition}
\newtheorem{lemma}{Lemma}
\newtheorem{corollary}{Corollary}

\renewcommand{\phi}{\varphi}

\renewcommand{\P}{\mathbb{P}}
\newcommand{\E}{\mathbb{E}}

\newcommand{\R}{\mathbb{R}}

\newcommand{\cD}{\mathcal{D}}

\newcommand{\cF}{\mathcal{F}}

\newcommand{\cE}{\mathcal{E}}
\newcommand{\cW}{\mathcal{W}}
\newcommand{\cU}{\mathcal{U}}

\def\ds1{\mathds{1}}
\renewcommand{\epsilon}{\varepsilon}

\newcommand\Var{{\dsV\text{ar}}\,}
\newcommand\dsV{\mathbb{V}}

\newcommand{\wh}{\widehat}

\newcommand{\argmin}{\mathop{\mathrm{argmin}}}

\renewcommand{\tilde}{\widetilde}

\newlength{\minipagewidth}
\newcommand{\beq}{\begin{equation}}
\newcommand{\eeq}{\end{equation}}

\newcommand{\beqa}{\begin{eqnarray}}
\newcommand{\eeqa}{\end{eqnarray}}

\newcommand{\beqan}{\begin{eqnarray*}}
\newcommand{\eeqan}{\end{eqnarray*}}

\def\ba#1\ea{\begin{align*}#1\end{align*}} %\ba = \begin{algin*}, \ea = \end{align*}
\def\banum#1\eanum{\begin{align}#1\end{align}} %\banum = \begin{algin}, \eanum

\newcommand{\UA}{\mathrm{UA}}
\newcommand{\TV}{\mathrm{TV}}

\newcommand{\diam}{\mathrm{diam}}
\newcommand{\dist}{\mathrm{dist}}

\newcommand{\utau}{\underline{\tau}}
\newcommand{\EE}{\mathbb{E}}
\newcommand{\uphi}{\underline{\phi}}
\newcommand{\upsi}{\underline{\psi}}

\newcommand{\ST}{G}
\newcommand{\usigma}{\underline{\sigma}}
\newcommand{\ua}{\underline{a}}

\newcommand{\BlackBox}{\rule{1.5ex}{1.5ex}}  % end of proof
\newenvironment{proof}{\par\noindent{\bf Proof\ }}{\hfill\BlackBox\\[2mm]}
\newtheorem{fact}{Fact}

\def \simless {\mathbin{\lower 0.1pt\hbox{$\char'074 {\hspace{-0.41cm}\raise 5.5pt \hbox{$\mathchar"7218$}\hspace{0.5pt}}$}}}

\begin{document}
\title{From trees to seeds: on the inference of the seed from large trees \\ in the uniform attachment model}
\author{
	S{\'e}bastien Bubeck
	\thanks{Microsoft Research and Princeton University; \texttt{sebubeck@microsoft.com}.}
	\and
	Ronen Eldan
	\thanks{Microsoft Research; \texttt{roneneldan@gmail.com}.}
	\and
	Elchanan Mossel
	\thanks{University of Pennsylvania and University of California, Berkeley; \texttt{mossel@wharton.upenn.edu}.}
	\and
	Mikl\'os Z.\ R\'acz
	\thanks{University of California, Berkeley; \texttt{racz@stat.berkeley.edu}.}
}
\date{\today}

\maketitle

%%%%%%%%%%%%%%%%
%%% Abstract %%%
%%%%%%%%%%%%%%%%

\begin{abstract}
 We study the influence of the seed in random trees grown according to the uniform attachment model, also known as uniform random recursive trees. 
 We show that different seeds lead to different distributions of limiting trees from a total variation point of view. 
 To do this, we construct statistics that measure, in a certain well-defined sense, global ``balancedness'' properties of such trees. 
 % Our results then follow by showing that the limiting distribution of these statistics is different for different seeds. 
 Our paper follows recent results on the same question for the preferential attachment model.
 % and introduces several new open problems.
\end{abstract}

%%%%%%%%%%%%%%%%
%%% Document %%%
%%%%%%%%%%%%%%%%

%%%%%%%%%%%%%%%%%%%%%%%%%%%%%%%%%%%%%%%%%%%%
\section{Introduction} \label{sec:intro} %%%
%%%%%%%%%%%%%%%%%%%%%%%%%%%%%%%%%%%%%%%%%%%%
We consider one of the simplest models of a randomly growing graph: starting from a single node, each arriving new node connects uniformly at random to one of the existing nodes. We denote by $\UA(n)$ the corresponding tree on $n$ vertices, also known as the uniform random recursive tree. We investigate whether a snapshot of the tree at a finite time can give information about the asymptotic future evolution of the tree. More precisely, we want to know if conditioning on the value of $\UA(k)$ for some finite $k$ affects the limiting distribution of $\UA(n)$ as $n$ goes to infinity. Equivalently, we study the uniform attachment model starting from an initial seed, and we want to know if the limiting distribution depends on the seed.

\medskip

For $n \geq k \geq 2$ and a tree $S$ on $k$ vertices, we define the random tree $\UA \left( n, S \right)$ by induction. 
First, $\UA \left( k, S \right) = S$. 
Then, given $\UA \left( n, S \right)$,  $\UA \left( n+1, S \right)$ is formed from $\UA \left( n, S \right)$ by adding a new vertex $u$ and adding a new edge $uv$ where the vertex $v$ is chosen uniformly at random among vertices of $\UA \left( n, S \right)$, independently of all past choices.  For two seed trees $S$ and $T$, we are interested in studying the quantity
\[
\delta \left( S, T \right) = \lim_{n \to \infty} \mathrm{TV} \left( \UA\left( n, S \right), \UA\left( n, T \right) \right), 
\]
a limit which is well-defined\footnote{This is because $\mathrm{TV}(\UA(n, S), \UA(n, T))$ is non-increasing in $n$ (since one can simulate the future evolution of the process) and always nonnegative.}, 
where we recall that the total variation distance between two random variables $X$ and $Y$ taking values in a finite space $\mathcal{X}$ with laws $\mu$ and $\nu$ is defined as $\TV \left(X,Y \right) = \frac{1}{2} \sum_{x\in \mathcal{X}} \left| \mu \left( x \right) - \nu \left( x \right) \right|$. 
Our main result shows that each seed leads to a unique limiting distribution of the uniform attachment tree. 
\begin{theorem}\label{thm:main}
For any $S$ and $T$ non-isomorphic with at least $3$ vertices, we have that $\delta(S,T) > 0$.
\end{theorem}

In some cases our method can say even more. 
As a proof of concept, we prove the following result, which states that the distance between a fixed tree and a star can be arbitrarily close to $1$ if the star is large enough.

\begin{theorem}\label{th:arbstars}
Let $S_k$ denote the $k$-vertex star. 
For any fixed tree $T$ one has
\[
\lim_{k \to \infty} \delta \left( S_{k}, T \right) = 1.
\]
\end{theorem}

%%%%%%%%%%%%%%%%%%%%%%%%%%%%%%%%%%%%%%%%%%%%%%%%%
\subsection{Our approach and comparison with the preferential attachment model} \label{sec:related} %%%
%%%%%%%%%%%%%%%%%%%%%%%%%%%%%%%%%%%%%%%%%%%%%%%%%

The theoretical study of the influence of the seed graph in growing random graph models was initiated, to the best of our knowledge, by~\cite{bubeck2014influencePAseed} for the particular case of preferential attachment trees. 
They showed that seeds with different degree profiles lead to different distributions of limiting trees from a total variation point of view. 
\cite{curien2014scaling}, using a different but related approach, then showed that this also holds for any two non-isomorphic seeds, i.e., that the analogue of Theorem~\ref{thm:main} above holds for the preferential attachment model. 

\medskip

The main idea in both papers, motivated by the rich-get-richer property of the preferential attachment model, is to consider various statistics based on large degree nodes, and show that the initial seed influences the distribution of these statistics. 
Consider, for instance, the problem of differentiating between the two seed trees in Figure~\ref{fig:ex1}. 
On the one hand, in $S$ the degree of $v_\ell$ is greater than that of $v_r$, and this unbalancedness in the degrees likely remains as the tree grows according to preferential attachment. 
On the other hand, in $T$ the degrees of $v_{\ell}$ and $v_r$ are the same, so they will have the same distribution at larger times as well. 
This difference in the balancedness vs.\ unbalancedness of the degrees of $v_{\ell}$ and $v_r$ is at the root of why  the seed trees $S$ and $T$ are distinguishable in the preferential attachment model. 
A precise understanding of the balancedness properties of the degrees relies on the classical theory of P\'olya urns. 

\begin{figure}[ht!]
 \centering
  \begin{subfigure}[h]{0.95\textwidth}
   \centering
   \includegraphics[width=0.7\textwidth]{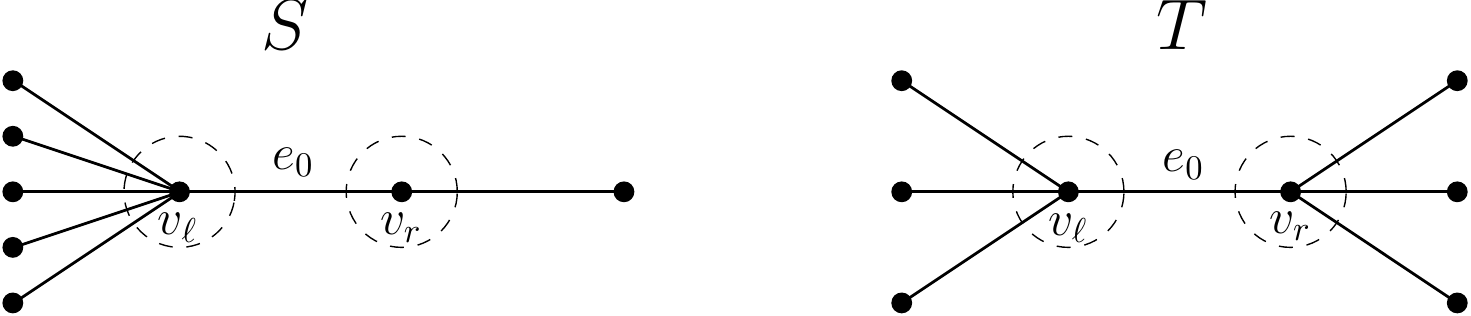}\\[0.25cm]
   \caption{\textbf{Preferential attachment.} The degrees of $v_{\ell}$ and $v_r$ are unbalanced in $S$ but balanced in $T$, 
and this likely remains the case as the trees grow according to preferential attachment. 
 This is at the root of why $S$ and $T$ are distinguishable as seed trees in the preferential attachment model.}
   \label{fig:ex1_pref}
  \end{subfigure}
  \\[0.25cm]
  \begin{subfigure}[h]{0.95\textwidth}
   \centering
   \includegraphics[width=0.75\textwidth]{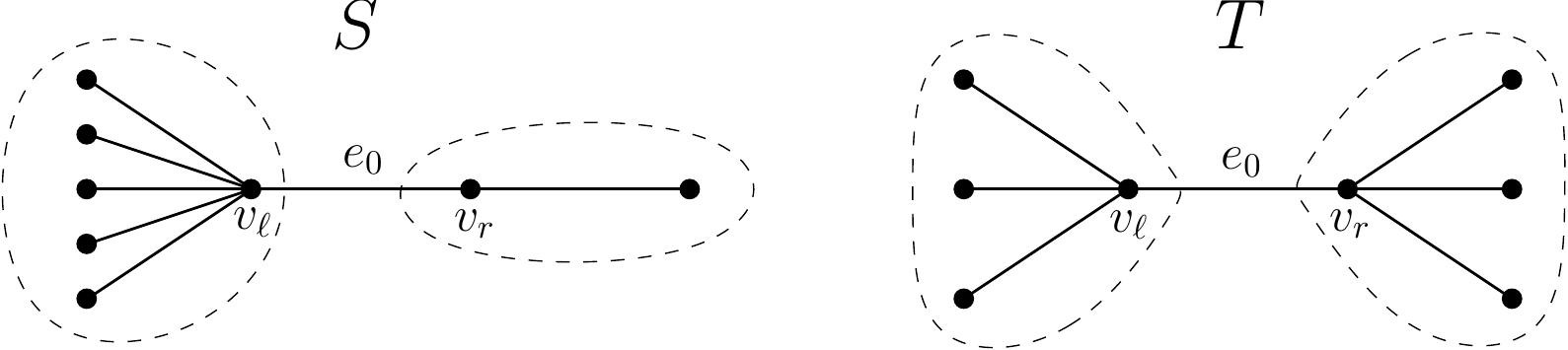}  \\[0.25cm]
   \caption{\textbf{Uniform attachment.} The sizes of the subtrees under $v_{\ell}$ and $v_r$ are unbalanced in $S$ but balanced in $T$, 
and this likely remains the case as the trees grow according to uniform attachment. 
This is at the root of why $S$ and $T$ are distinguishable as seed trees in the uniform attachment model.}
   \label{fig:ex1_uni}
  \end{subfigure}
  \caption{Distinguishing between two trees requires different approaches for the uniform and the preferential attachment models.}
  \label{fig:ex1}
\end{figure}

\medskip 

In the uniform attachment model the degrees of vertices do not play a special role. 
In particular, in the example of Figure~\ref{fig:ex1}, $v_\ell$ and $v_r$ will have approximately similar degrees in a large tree grown according to the uniform attachment model, irrespective of whether the seed tree is $S$ or $T$. 
Nonetheless, we are able to distinguish the seed trees $S$ and $T$, but the statistics we use to do this are based on more global balancedness properties of these trees. 

\medskip

An edge of a tree partitions the tree into two parts on either side of the edge. 
For most edges in a tree, this partition has very unbalanced sizes; for instance, if an edge is adjacent to a leaf, then one part contains only a single vertex. 
On the other hand, for edges that are in some sense ``central'' the partition is more balanced, in the sense that the sizes of the two parts are comparable. 
Intuitively, the edges of the seed tree will be among the ``most central'' edges of the uniform attachment tree at large times, 
and so we expect that the seed should influence the global balancedness properties of such trees.

\medskip

Consider again the example of the two seed trees $S$ and $T$ in Figure~\ref{fig:ex1}. 
The edge $e_0$ partitions the tree into two parts: a subtree under $v_{\ell}$ and a subtree under $v_r$. 
In $S$ these subtree sizes are unbalanced, and this likely remains the case as the tree grows according to uniform attachment. 
On the other hand, in $T$ the subtree sizes are equal, and they will likely remain balanced as the tree grows. 
Again, P\'olya urns play an important role, since the subtree sizes evolve according to a classical P\'olya urn initialized by the subtree sizes in the seed tree. 
The difference in the balancedness vs.\ unbalancedness of the subtree sizes is at the root of why $S$ and $T$ are distinguishable in the uniform attachment model. 
To prove Theorem~\ref{thm:main} we need to analyze statistics based on more general global balancedness properties of such trees, but the underlying intuition is what is described in the preceding paragraphs. 

\medskip

To formalize this intuition we essentially follow the proof scheme developed in~\cite{curien2014scaling}. 
However, the devil is in the details: since the underlying statistics are markedly different---in particular, statistics based on degrees are local, whereas those based on balancedness properties of subtree sizes are global---some of the essential steps of the proof become different. 
We provide a more detailed comparison to the work of~\cite{curien2014scaling} in Section~\ref{sec:curien}, after we present our proof. 

%%%%%%%%%%%%%%%%%%%%%%%%%%%%%%%%%%%%%%%%%%%%%%%%%
\subsection{Further related work} \label{sec:more_related} %%%
%%%%%%%%%%%%%%%%%%%%%%%%%%%%%%%%%%%%%%%%%%%%%%%%%

A tree with node set $\left[ n \right]:= \left\{ 1, \dots, n \right\}$ is called recursive if the node numbers along the unique path from $1$ to $j$ increase for every $j \in \left\{ 2, \dots, n \right\}$. 
A stochastic process of random growing trees where nodes are labeled according to the time they are born is thus a sequence of random recursive trees. 
If we choose a recursive tree with node set $\left[ n \right]$ uniformly at random, 
then the resulting tree has the same distribution as 
%$\UA\left(n, S_1 \right)$, i.e., a tree grown according to the uniform attachment process starting from a single node. 
$\UA \left( n \right)$. 
A random tree grown according to the preferential attachment process starting from a single node is also known as a random plane-oriented recursive tree, see~\cite{Mah92}. 

\medskip

There is a large literature on random recursive trees and their various statistics; we refer the reader to the book by~\cite{drmota2009random}. 
Of particular interest to the question we study here are recent works on a boundary theory approach to the convergence of random recursive trees in the limit as the tree size goes to infinity, see~\cite{evans2012trickle} and~\cite{grubel2014random}. 
The main difference between these and the current work is that they consider labeled and rooted trees, 
whereas we are interested in 
what can be said about the seed given an unlabeled and unrooted copy of the tree.

%%%%%%%%%%%%%%%%%%%%%%%%%%%%%%%%%%%%%%%%%%%%%%%%%%%%%%%%%%%%%%%%%%%%%
\section{Partitions and their balancedness: simple examples} \label{sec:intuition} %%%
%%%%%%%%%%%%%%%%%%%%%%%%%%%%%%%%%%%%%%%%%%%%%%%%%%%%%%%%%%%%%%%%%%%%%

In this section we show on a simple example how to formalize the intuition described in Section~\ref{sec:related}. 
We define a simple statistic based on this intuition, and after collecting some preliminary facts in Section~\ref{sec:prelim}, we show in Section~\ref{sec:ex} that $\delta \left( P_4, S_4 \right) > 0$, where $P_4$ and $S_4$ are the path and the star on four vertices, respectively. 
We conclude the section by proving Theorem~\ref{th:arbstars} in Section~\ref{sec:stars}. 
The goal of this section is thus to provide a gentle introduction into the methods and statistics used to distinguish different seed trees, 
before analyzing more general statistics in the proof of Theorem~\ref{thm:main} in Section~\ref{sec:proofs}.

\medskip

For a tree $T$ and an edge $e \in E(T)$, let $T_1$ and $T_2$ be the two connected components of $T \setminus \{e\}$. Define
$$
g(T,e) = |T_1|^2|T_2|^2 / |T|^4,
$$
where $|T|$ denotes the number of vertices of $T$. 
Clearly, $0 \leq g \left( T, e \right) \leq 1/16$, and for ``peripheral'' edges $e$, $g \left( T, e \right)$ is closer to $0$, while for more ``central'' edges $e$, $g \left( T, e \right)$ is closer to $1/16$. 
Define the following statistic:
$$
\ST(T) = \sum_{e \in E(T)} g(T,e).
$$
The statistic $\ST\left( T\right)$ thus measures in a particular way the global balancedness properties of the tree $T$, and ``central'' edges contribute the most to this statistic.

\subsection{Preliminary facts}\label{sec:prelim}

For all $\alpha,\beta,n \in \mathbb{N}$, let $B_{\alpha,\beta,n}$ be a random variable such that $B_{\alpha, \beta, n} - \alpha$ has the beta-binomial distribution with parameters $(\alpha,\beta,n)$, i.e., it is a random variable satisfying
\[
\P\left(B_{\alpha,\beta,n} = \alpha + k\right) = \frac{(k+\alpha-1)! (n-k+\beta - 1)! (\alpha + \beta - 1)!}{ (n+\alpha+\beta - 1)! (\alpha - 1)! (\beta - 1)! }  \binom{n}{k}, ~~ \forall k \in \left\{0, 1, \dots, n \right\}.
\]
The key to understanding the statistic $\ST$ is the following distributional identity:
\begin{equation} \label{betabinomial}
g\left(\UA(n,S), e\right) \stackrel{d}{=} \frac{1}{n^4} B_{|T_1|,|T_2|,n-|S|}^2(n-B_{|T_1|,|T_2|,n-|S|})^2, ~~ \forall e \in E \left( S \right),
\end{equation}
where $T_1$ and $T_2$ are defined, given $e$, as above, and $\stackrel{d}{=}$ denotes equality in distribution. 
This is  an immediate consequence of the characterization of $\left( B_{\alpha,\beta,n}, n + \left( \alpha + \beta \right) - B_{\alpha, \beta, n} \right)$ as the distribution of a classical P\'olya urn with replacement matrix $\left(\begin{smallmatrix} 1 & 0 \\ 0 & 1 \end{smallmatrix}\right)$ and starting state $(\alpha,\beta)$ after $n$ draws. 
Similarly, for edges not in the seed $S$ we have
\begin{equation}\label{eq:betabinomial_late}
g \left(\UA(n,S), e_j \right) \stackrel{d}{=} \frac{1}{n^4} B_{1,j,n-j-1}^2 \bigl (n-B_{1,j,n-j-1} \bigr )^2, 
\end{equation}
where $e_j \in E \left( \UA \left(j+1, S \right) \right) \setminus E \left( \UA \left(j, S \right) \right)$, $j \in \left\{ \left| S \right|, \dots, n - 1 \right\}$.

\medskip

We use the following elementary facts about the beta-binomial distribution, which we prove in Appendix~\ref{app:beta}:
\begin{fact} \label{factbbmoments}
For every $p \geq 1$ there exists a constant $C(p)$ such that for all $\alpha, \beta, n$ such that $n \geq \alpha + \beta$, we have
\begin{equation} \label{eq:moments}
\left ( \EE[ B_{\alpha,\beta,n-\alpha-\beta}^p ] \right )^{1/p}  \leq C(p) n \frac{\alpha}{\alpha + \beta}.
\end{equation}
\end{fact}
\begin{fact} \label{factbbsmall}
There exists a universal constant $C>0$ such that whenever $\alpha, \beta \geq 1$, $n \geq \alpha + \beta$, and $t \geq 0$, we have
\begin{equation} \label{eq:fact2}
 \P \left ( B_{\alpha,\beta,n-\alpha-\beta} < t n \frac{\alpha}{\alpha + \beta} \right ) \leq C t.
\end{equation}
\end{fact}

\subsection{A simple example}\label{sec:ex}

After these preliminaries we are now ready to show that $\delta \left( P_4, S_4 \right) > 0$. 
To abbreviate notation, in the following we write simply $P \equiv P_4$ and $S\equiv S_4$. 
In order to show that $\delta \left( P, S \right) > 0$, it is enough to show two things:
\begin{equation} \label{eq:tvexp}
\liminf_{n \to \infty } \left| \EE \left[\ST\left(\UA\left(n, P\right)\right)\right] - \EE\left[\ST\left(\UA\left(n, S\right)\right)\right] \right| > 0,
\end{equation}
and
\begin{equation} \label{eq:tvvar}
\limsup_{n \to \infty } \left( \Var[\ST(\UA(n, P))] + \Var[\ST(\UA(n, S))] \right) < \infty.
\end{equation}
The proof can then be concluded using the Paley-Zigmund inequality (for more detail on this point, see the proof of Theorem~\ref{thm:main} in Section~\ref{sec:obs}). 

\medskip

For $j \geq 4$, let $e_j^P$ denote the edge in $\UA \left( j+1, P \right) \setminus \UA \left( j, P \right)$, and define $e_j^S$ similarly. 
Towards~\eqref{eq:tvexp}, we first observe that
\[
 g \left( \UA \left( n, P \right) , e_j^P \right) \stackrel{d}{=} g \left( \UA \left( n, S \right) , e_j^S \right), ~~ \forall j \in \left\{4, \dots, n - 1 \right\}.
\]
Consequently, we have
\[
\EE \left[\ST \left(\UA \left(n, P \right) \right) \right] - \EE \left[\ST \left(\UA \left(n, S \right) \right) \right] = \sum_{e \in P} \EE \left[g \left( \UA \left(n,P \right), e \right) \right] - \sum_{e \in S} \EE \left[g \left( \UA \left(n,S\right), e \right)\right].
\]
Moreover, note that $P$ has two edges, $e_1$ and  $e_2$, such that $P \setminus \{e_i\}$ has two connected components of sizes $1$ and $3$ for $i=1,2$. Since this is true for all edges of the star $S$, we conclude that
\[
\EE[\ST(\UA(n, P))] - \EE[\ST(\UA(n, S))] = \EE[g(\UA(n,P), e_3)] - \EE[g(\UA(n,S), e_1)],
\]
where $e_3$ is the remaining edge of $P$ (i.e., the middle edge of the path). Using~\eqref{betabinomial}, we thus have
\begin{multline*}
 \EE[\ST(\UA(n, P))] - \EE[\ST(\UA(n, S))] \\
\begin{aligned} 
 &= \frac{1}{n^4} \left(\EE \left[ B_{2,2,n-4}^2 \left(n - B_{2,2,n-4} \right)^2 \right] - \EE \left[ B_{1,3,n-4}^2 \left(n - B_{1,3,n-4} \right)^2  \right] \right) \\
&= \frac{2n^3+5n^2+8n+5}{140n^3},
\end{aligned}
\end{multline*}
where the last equality is attained via a straightforward calculation using explicit formulae for the first four moments of the beta-binomial distribution. We see that 
\[
\lim_{n \to \infty} \left( \EE[\ST(\UA(n, P))] - \EE[\ST(\UA(n, S))] \right) = \frac{1}{70} \neq 0,
\]
which establishes~\eqref{eq:tvexp}. 

\medskip

It remains to prove~\eqref{eq:tvvar}. We show now that $\limsup_{n \to \infty} \Var[\ST(\UA(n, P))] < \infty$; the proof that $\limsup_{n \to \infty} \Var[\ST(\UA(n, S))] < \infty$ is identical. To abbreviate notation, write $T_n$ for $\UA \left(n, P \right)$. 
Similarly as above, for $j \geq 4$ let $e_j$ be the edge in $T_{j+1} \setminus T_j$, and let $e_1$, $e_2$, and $e_3$ be the edges of $P$ in some arbitrary order. 
Using Cauchy-Schwarz we have that 
\begin{equation} \label{eq:cs}
\Var[\ST(T_n)] \leq \left ( \sum_{j=1}^{n-1} \sqrt{\Var[g(T_n, e_j)]} \right )^2.
\end{equation}
For any edge $e_i$ we clearly have $0 \leq g(T_n,e_i) \leq 1$, and so 
\begin{equation} \label{eqfirstterms}
\sum_{j=1}^3 \sqrt{\Var[g(T_n, e_j)]} \leq 3.
\end{equation}
Next, fix $4 \leq j \leq n-1$. Using formula~\eqref{eq:betabinomial_late} we know that
$$
g(T_{n}, e_j) \stackrel{d}{=} \frac{1}{n^4} B_{1,j,n-j-1}^2(n-B_{1,j,n-j-1})^2.
$$
The estimate~\eqref{eq:moments} yields $\EE[B_{1,j,n-j-1}^4] \leq C n^4 / j^4$, 
where $C>0$ is a universal constant. Consequently, we have 
$$
\EE[g(T_{n}, e_j)^2] \leq C/j^4, ~~ \forall j \in \left\{ 4, \dots, n-1 \right\},
$$
which, in turn, implies that 
$$
\sqrt{\Var[g(T_n, e_j)]} \leq C/j^2.
$$
Plugging this inequality and~\eqref{eqfirstterms} into~\eqref{eq:cs} establishes~\eqref{eq:tvvar}. This completes the proof of $\delta \left( P_4, S_4 \right) > 0$. 

\medskip

The statistic $\ST \left( \cdot \right)$ cannot distinguish between all pairs of non-isomorphic trees; however, appropriate generalizations of it can. 
An alternative description of $\ST \left( \cdot \right)$ is as follows. 
Let $\tau$ be a tree consisting of two vertices connected by a single edge. 
Up to normalization, the quantity $G\left( T \right)$ is equal to the sum over all embeddings $\phi : \tau \to T$ 
of the product of the squares of the connected components of $T \setminus \phi \left( \tau \right)$. 
A natural generalization of this definition is to take $\tau$ to be an arbitrary finite tree. 
Moreover, we can assign natural numbers to each vertex of $\tau$, which determines the power to which we raise the size of the respective connected components. 
In this way we obtain a family of statistics associated with so-called \emph{decorated trees}. 
It turns that this generalized family of statistics can indeed distinguish between any pair of non-isomorphic trees; for details see Section~\ref{sec:proofs}.

\subsection{Distinguishing large stars: a proof of Theorem~\ref{th:arbstars}}\label{sec:stars}

In the following we first give an upper bound on the probability that $G \left( \UA \left(n, T \right) \right)$ is small, and then we give an upper bound on the probability that $G\left( \UA \left( n, S_k \right) \right)$ is not too small. The two together will prove Theorem~\ref{th:arbstars}.

\medskip

First, fix a tree $T$ and choose an arbitrary edge $e_1 \in E \left( T \right)$. 
Let $T_n'$ and $T_n''$ be the two connected components of $\UA(n, T) \setminus \{e_1\}$, 
defined consistently such that $T_j' \subset T_{j+1}'$ and $T_j'' \subset T_{j+1}''$ (and otherwise the order is chosen arbitrarily). 
We have
\[
 g(\UA(n,T), e_1) = \frac{1}{n^4} |T_n'|^2 |T_n''|^2 = \frac{1}{n^4} |T_n'|^2 (n - |T_n'|)^2.
\]
By equation~\eqref{betabinomial} we have 
\[
 |T_n'| \stackrel{d}{=} B_{a,|T|-a,n-|T|}^2,
\]
where $a := |T_{|T|}'| \geq 1$. Using Fact~\ref{factbbsmall} we then have that for all $t > 0$, 
\begin{equation}\label{eq:bd_near_zero}
 \P \left ( \frac{|T_n'|^2}{n^2} < t \frac{1}{|T|^2}  \right ) \leq \P \left ( \frac{|T_n'|^2}{n^2} < t \frac{a^2}{|T|^2}  \right ) = \P \left ( \frac{B_{a,|T|-a,n-|T|}^2}{n^2} < t \frac{a^2}{|T|^2}  \right ) \leq C \sqrt{t}.
\end{equation}
Consider the event $E_n = \{|T_n'| \leq n/2 \}$ and note that if $E_n$ holds then $g \left( \UA \left( n, T \right), e_1 \right) \geq \left| T_n' \right|^2 / \left( 4n^2 \right)$. 
This, together with~\eqref{eq:bd_near_zero}, gives
\[
 \P \left( E_n \cap \left\{ g \left( \UA \left(n, T \right), e_1 \right) < t \frac{1}{4\left| T \right|^2} \right\} \right) \leq C \sqrt{t}, ~~ \forall t > 0.
\]
By repeating the above argument with $T_n''$ instead of $T_n'$, we also have
\[
 \P \left( E_n^C \cap \left\{ g \left( \UA \left(n, T \right), e_1 \right) < t \frac{1}{4\left| T \right|^2} \right\} \right) \leq C \sqrt{t}, ~~ \forall t > 0,
\]
and thus we conclude that
\[
 \P \left( g \left( \UA \left(n, T \right), e_1 \right) < t \frac{1}{4 \left| T \right|^2} \right) \leq 2 C \sqrt{t}, ~~ \forall t > 0.
\]
Now since $\ST \left( \UA \left(n, T \right) \right) \geq g \left( \UA \left( n, T \right), e_1 \right)$, and by setting $z = t / \left( 4 \left| T \right|^2 \right)$, we finally get that
\begin{equation}\label{STlowerbound}
 \P \left( \ST \left( \UA \left( n, T \right) \right) < z \right) \leq 4 C \sqrt{z} \left| T \right|, ~~ \forall z > 0.
\end{equation}

\medskip

In order to understand the distribution of $\ST \left( \UA \left( n, S_k \right) \right)$ we first estimate its mean:
\begin{multline*}
\E \left[ \ST \left( \UA \left( n, S_k \right) \right) \right] \\
\begin{aligned}
 &= \frac{k-1}{n^4} \E \left[ B_{1,k-1,n-k}^2 \left( n - B_{1,k-1,n-k} \right)^2 \right] + \sum_{j=k}^{n-1} \frac{1}{n^4} \E \left[ B_{1,j,n-j-1}^2 \left( n - B_{1,j,n-j-1} \right)^2 \right] \\
 &\leq \frac{k-1}{n^2} \E \left[ B_{1,k-1,n-k}^2  \right] + \frac{1}{n^2} \sum_{j=k}^{n-1} \E \left[ B_{1,j,n-j-1}^2 \right] \stackrel{\eqref{eq:moments}}{\leq} \frac{C'}{k} + \sum_{j=k}^{n-1} \frac{C'}{j^2} \leq \frac{3C'}{k}
\end{aligned}
\end{multline*}
for some absolute constant $C'$. 
Now using Markov's inequality with this estimate, and also taking $z = 3C' / \sqrt{k}$ in the inequality~\eqref{STlowerbound}, we get that 
\[
 \P \left( G \left( \UA \left( n, S_k \right) \right) \geq 3C' / \sqrt{k} \right) \leq \frac{1}{\sqrt{k}} \qquad \text{and} \qquad  \P \left( G \left( \UA \left( n, T \right) \right) < 3C' / \sqrt{k} \right) \leq \frac{C''}{k^{1/4}}
\]
for some absolute constant $C''$. This then immediately implies that $\delta \left( S_k, T \right) \to 1$ as $k \to \infty$.

%%%%%%%%%%%%%%%%%%%%%%%%%%%%%%%%%%%%%%%%%%%%%%%%%%%%%%%%%%%%%%%%
\section{Proof of Theorem~\ref{thm:main}} \label{sec:proofs} %%%
%%%%%%%%%%%%%%%%%%%%%%%%%%%%%%%%%%%%%%%%%%%%%%%%%%%%%%%%%%%%%%%%

After the intuition and simple examples provided in Section~\ref{sec:intuition}, in this section we fully prove Theorem~\ref{thm:main}. 
As mentioned in Section~\ref{sec:related}, the proof shares some features with the proof in~\cite{curien2014scaling} for preferential attachment, but is different in several ways. 
These differences are discussed in detail after the proof, in Section~\ref{sec:curien}.

\medskip

\textbf{Notation.} 
For a graph $G$, denote by $V \left( G \right)$ the set of its vertices, by $E \left( G \right)$ the set of its edges, and by $\diam \left( G \right)$ its diameter. For brevity, we often write $v \in G$ instead of $v \in V \left(G \right)$.  
For integers $k,j \geq 1$, define the descending factorial $\left[k \right]_j = k \left( k - 1 \right) \dots\left( k - j + 1 \right)$, and also let $\left[k \right]_0 = 1$. 
For two sequences of real numbers $\left\{ a_n \right\}_{n \geq 0}$ and $\left\{ b_n \right\}_{n \geq 0}$, we write $a_n \simless b_n$ (to be read as $a_n$ is less than $b_n$ up to $\log$ factors) if there exist constants $c > 0$, $\gamma \in \R$, and $n_0$ such that $\left| a_n \right| \leq c \left( \log \left( n \right) \right)^{\gamma} \left| b_n \right|$ for all $n \geq n_0$. 
For a sequence $\left\{ a_n \right\}_{n \geq 0}$ of real numbers, define $\Delta_n a = a_{n+1} - a_n$ for $n \geq 0$.

\subsection{Decorated trees}\label{sec:dec_trees}

A \emph{decorated tree} is a pair $\utau = \left( \tau, \ell \right)$ consisting of a tree $\tau$ and a family of nonnegative integers $\left( \ell \left( v \right) ; v \in \tau \right)$, called labels, associated with its vertices; see Figure~\ref{fig:dec_tree} for an illustration. 
Let $\cD$ denote the set of all decorated trees, 
$\cD_{+}$ the set of all decorated trees where every label is positive, 
$\cD_{0}$ the set of all decorated trees where there exists a label which is zero, and 
finally let $\cD_0^*$ denote the set of all decorated trees where there exists a leaf which has label zero. 

\medskip

Define $\left| \utau \right|$ to be the number of vertices of $\tau$, and let $w\left( \utau \right) := \sum_{v \in \tau} \ell \left( v \right)$ denote the total weight of $\utau$. 
For $\utau, \utau' \in \cD$, let $\utau \prec \utau'$ if $\left| \utau \right| < \left| \utau' \right|$ and $w \left( \utau \right) \leq w \left( \utau' \right)$ or $\left| \utau \right| \leq \left| \utau' \right|$ and $w \left( \utau \right) < w \left( \utau' \right)$. 
This defines a strict partial order $\prec$, and let $\preccurlyeq$ denote the associated partial order, 
i.e., $\utau \preccurlyeq \utau'$ if and only if $\utau \prec \utau'$ or $\utau = \utau'$.

\medskip

For $\utau \in \cD$, let $L \left( \utau \right)$ denote the set of leaves of $\tau$, let $L_0 \left( \utau \right) := \left\{ v \in L \left( \utau \right) : \ell \left( v \right) = 0 \right\}$, $L_1 \left( \utau \right) := \left\{ v \in L \left( \utau \right) : \ell \left( v \right) = 1 \right\}$, and $L_{0,1} \left( \utau \right) := L_0 \left( \utau \right) \cup L_1 \left( \utau \right)$. 
For $\utau \in \cD$ and $v \in L \left( \utau \right)$, define $\utau_v$ to be the same as $\utau$ except the leaf $v$ and its label are removed. 
For $\utau \in \cD$ and a vertex $v \in \utau$ such that $\ell \left( v \right) \geq 2$, define $\utau_v'$ to be the same as $\utau$ except the label of $v$ is decreased by one, i.e., $\ell_{\utau_v'} \left( v \right) = \ell_{\utau} \left( v \right) - 1$.

\begin{figure}[h!]
  \centering
    \includegraphics[width=0.67\textwidth]{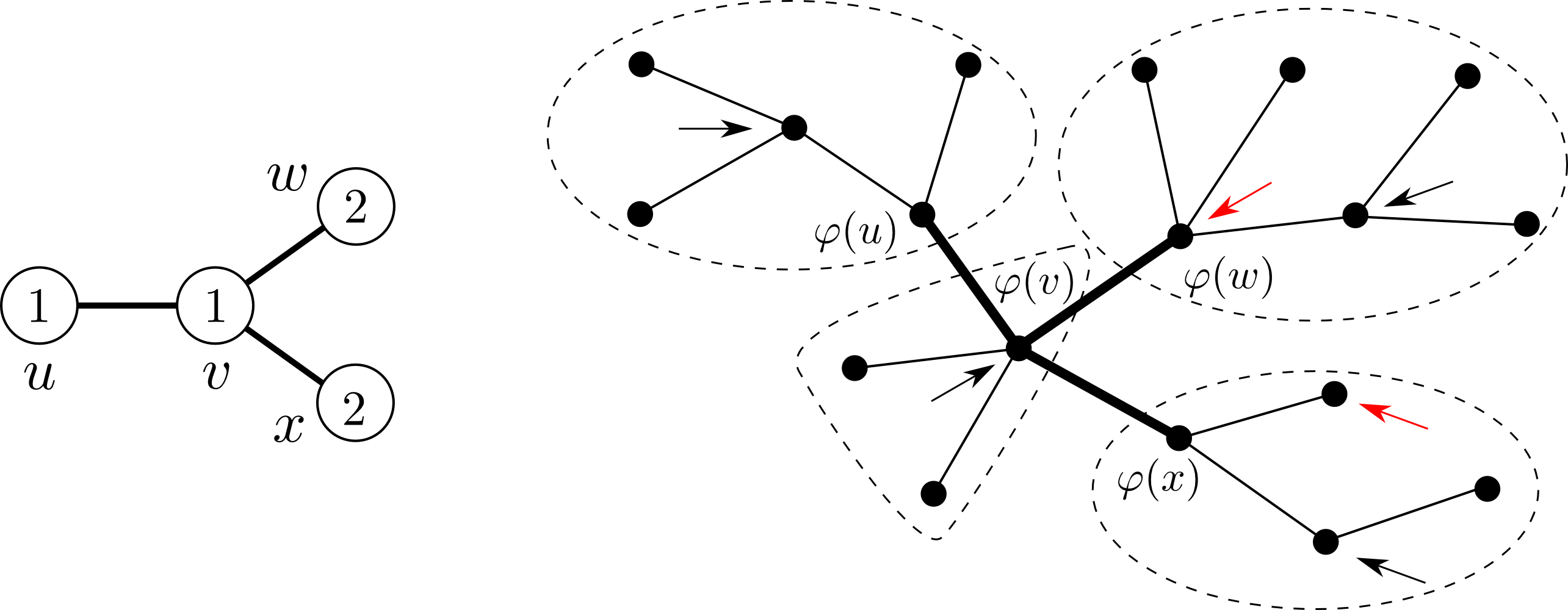}
  \caption{\textbf{A decorated tree and a decorated embedding.} 
On the left is a decorated tree $\utau = \left( \tau, \ell \right)$ with four vertices, two of them having label $1$, and two of them having label $2$. 
On the right is a larger tree $T$, and an embedding $\phi : \tau \to T$ depicted in bold. 
The connected components of the forest $\wh{T} \left( T, \tau, \phi \right)$ are circled with dashed lines, with the component sizes being $f_{\phi\left( u \right)} \left( T \right) = 5$, $f_{\phi\left( v \right)} \left( T \right) = 3$, $f_{\phi\left( w \right)} \left( T \right) = 6$, and $f_{\phi\left( x \right)} \left( T \right) = 4$. 
A decorated embedding $\uphi$ is also depicted, which consists of the embedding $\phi$ together with the mapping of $w\left( \utau \right) = 6$ arrows to vertices of $T$. 
The arrows in each subtree are distinguishable, which is why they are depicted using different colors.}
  \label{fig:dec_tree}
\end{figure}

\subsection{Statistics and distinguishing martingales}\label{sec:obs}

Given two trees $\tau$ and $T$, a map $\phi : \tau \to T$ is called an \emph{embedding} if $\phi$ is an injective graph homomorphism. That is, $\phi$ is an injective map from $V\left( \tau \right)$ to $V \left( T \right)$ such that $\left\{ u,v \right\} \in E \left( \tau \right)$ implies that $\left\{ \phi \left( u \right), \phi \left( v \right) \right\} \in E \left( T \right)$. 

\medskip

For two trees $\tau$ and $T$, and an embedding $\phi : \tau \to T$, denote by $\wh{T} = \wh{T} \left( T, \tau, \phi \right)$ the forest obtained from $T$ by removing the images of the edges of $\tau$ under the embedding $\phi$; see Figure~\ref{fig:dec_tree} for an illustration. 
Note that the forest $\wh{T}$ consists of exactly $\left|\tau\right|$ trees, and each tree contains exactly one vertex which is the image of a vertex of $\tau$ under $\phi$. 
For $v \in \tau$, denote by $f_{\phi\left( v \right)} \left( T \right)$ the number of vertices of the tree in $\wh{T}$ which contains the vertex $\phi \left( v \right)$. 
Using this notation, for a decorated tree $\utau \in \cD$ define
\begin{equation}\label{eq:main_stat}
 F_{\utau} \left( T \right) = \sum_{\phi} \prod_{v \in \tau} \left[ f_{\phi \left( v \right)} \left( T \right) \right]_{\ell \left( v \right)},
\end{equation}
where the sum is over all embeddings $\phi : \tau \to T$. 
If there are no such embeddings then $F_{\utau} \left( T \right) = 0$ by definition. 
Note that if $\utau$ consists of a single vertex with label $k \geq 0$, then $F_{\utau} \left( T \right) = |T| \times \left[ \left| T \right| \right]_k$.

\medskip

The quantity $F_{\utau} \left( T \right)$ has a combinatorial interpretation which is useful: it is the number of \emph{decorated embeddings} of $\utau$ in $T$, defined as follows. 
Imagine that for each vertex $v \in \tau$ there are $\ell \left( v \right)$ distinguishable (i.e., ordered) arrows pointing to $v$. 
A decorated embedding $\uphi$ is an embedding $\phi$ of $\tau$ in $T$, together with a mapping of the arrows to vertices of $T$ in such a way that each arrow pointing to $v \in \tau$ is mapped to a vertex in the tree of $\wh{T}$ that contains $\phi \left( v \right)$, with distinct arrows mapped to distinct vertices. 
See Figure~\ref{fig:dec_tree} for an illustration.

\medskip

The quantities $F_{\utau}$ are also more amenable to analysis than other statistics, because their expectations satisfy recurrence relations, as described in Section~\ref{sec:recurrence}. 
Using the statistics $F_{\utau}$ it is possible to create martingales that distinguish between different seeds.

\begin{proposition}\label{prop:mg}
 Let $\utau \in \cD_+$. There exists a family of constants $\left\{ c_n \left( \utau, \utau' \right) : \utau' \in \cD_+, \utau' \preccurlyeq \utau, n \geq 2 \right\}$ with $c_n \left( \utau, \utau \right) > 0$ such that for every seed tree $S$, the process $\left\{ M_{\utau}^{\left( S \right)} \left( n \right) \right\}_{n \geq \left|S\right|}$ defined by
\[
 M_{\utau}^{\left( S \right)} \left( n \right) = \sum_{\utau' \in \cD_+ : \utau' \preccurlyeq \utau} c_n \left( \utau, \utau' \right) F_{\utau'} \left( \UA \left( n, S \right) \right)
\]
is a martingale with respect to the natural filtration $\cF_n = \sigma \left\{ \UA \left( |S|, S \right), \dots, \UA \left( n, S \right) \right\}$ and is bounded in $L^2$.
\end{proposition}
Note that in the construction of these martingales we only use decorated trees where every label is positive. 
As we shall see, we analyze decorated trees having a label which is zero in order to show that the martingales above are bounded in $L^2$. 
See Sections~\ref{sec:moment} and~\ref{sec:curien} for more details and discussion on this point.

\medskip

We now prove Theorem~\ref{thm:main} using Proposition~\ref{prop:mg}, which we then prove in the following subsections.

\medskip

\begin{proof}\textbf{of Theorem~\ref{thm:main}}
 Let $S$ and $T$ be two non-isomorphic trees with at least three vertices, and let $n_0 := \max \left\{ \left|S\right|, \left| T \right| \right\}$. 
First we show that there exists $\utau \in \cD_+$ such that 
\begin{equation}\label{eq:mean_diff}
 \E \left[ F_{\utau} \left( \UA \left( n_0, S \right) \right) \right] \neq \E \left[ F_{\utau} \left( \UA \left( n_0, T \right) \right) \right].
\end{equation}
Assume without loss of generality that $\left|S\right| \leq \left| T \right|$, and let $\utau$ be equal to $T$ with labels $\ell \left( v \right) = 1$ for all $v\in T$. 
Then for every tree $T'$ with $\left| T' \right| = \left| T \right|$ we have $F_{\utau} \left( T' \right) = F_{\utau} \left( T \right) \times \mathbf{1}_{\left\{ T' = T \right\}}$,
since if $T'$ and $T$ are non-isomorphic then there is no embedding of $T$ in $T'$. 
Note also that $F_{\utau} \left( T \right)$ is the number of automorphisms of $T$, which is positive. 
Consequently we have
\[
 \E \left[ F_{\utau} \left( \UA\left(n_0, S\right) \right) \right] = F_{\utau} \left( T \right) \times \P \left[ \UA \left( n_0, S \right) = T \right].
\]
When $\left| S \right| = \left| T \right|$, we have $\P \left[ \UA \left( n_0, S \right) = T \right] = 0$. 
When $\left| S \right| < \left| T \right|$ it is easy to see that the isomorphism class of $\UA \left( n_0, S \right)$ is nondeterministic 
(here we use the fact that $\left|S\right| \geq 3$), and so $\P \left[ \UA \left( n_0, S \right) = T \right] < 1$. 
In both cases we have that~\eqref{eq:mean_diff} holds. 

\medskip

Now let $\utau \in \cD_+$ be a minimal (for the partial order $\preccurlyeq$ on $\cD_+$) decorated tree for which~\eqref{eq:mean_diff} holds. 
By definition we then have that $\E \left[ F_{\utau'} \left( \UA \left( n_0, S \right) \right) \right] = \E \left[ F_{\utau'} \left( \UA \left( n_0, T \right) \right) \right]$ for every $\utau' \in \cD_+$ such that $\utau' \prec \utau$. 
By the construction of the martingales in Proposition~\ref{prop:mg} we then have that
\[
 \E \left[ M_{\utau}^{\left( S \right)} \left( n_0 \right) \right] \neq \E \left[ M_{\utau}^{\left( T \right)} \left( n_0 \right) \right].
\]
Clearly for any $n \geq n_0$ we have that $\TV \left( \UA \left( n, S \right), \UA \left( n, T \right) \right) \geq \TV \left( M_{\utau}^{\left( S \right)} \left( n \right), M_{\utau}^{\left( T \right)} \left( n \right) \right)$. 
Now let $\left( X, Y \right)$ be a coupling of $\left( M_{\utau}^{\left( S \right)} \left( n \right), M_{\utau}^{\left( T \right)} \left( n \right) \right)$. 
We need to bound from below $\P \left( X \neq Y \right)$ in order to obtain a lower bound on $\TV \left( M_{\utau}^{\left( S \right)} \left( n \right), M_{\utau}^{\left( T \right)} \left( n \right) \right)$. 
Using Paley-Zigmund's inequality we have that
\[
 \P \left( X \neq Y \right) \geq \frac{\left( \E \left[ \left| X - Y \right| \right] \right)^2}{\E \left[ \left( X - Y \right)^2 \right]}.
\]
By Jensen's inequality we have that $\left( \E \left[ \left| X - Y \right| \right] \right)^2 \geq \left( \E \left[ X \right] - \E \left[ Y \right] \right)^2$, and furthermore
\begin{align*}
 \E \left[ \left( X - Y \right)^2 \right] &= \E \left[ \left( X - \E \left[ X \right] + \E \left[ X \right] - \E \left[ Y \right] + \E \left[ Y \right] - Y \right)^2 \right] \\
 &= \Var \left( X \right) + \Var \left( Y \right) + \left( \E \left[ X \right] - \E \left[ Y \right] \right)^2 + 2 \E \left[ \left( X - \E \left[ X \right] \right) \left( \E \left[ Y \right] - Y \right) \right] \\
 &\leq 2 \Var \left( X \right) + 2 \Var \left( Y \right) + \left( \E \left[ X \right] - \E \left[ Y \right] \right)^2.
\end{align*}
Thus we have shown that
\begin{multline*}
 \TV \left( \UA \left( n, S \right), \UA \left( n, T \right) \right) \\
\geq \frac{\left( \E \left[ M_{\utau}^{\left( S \right)} \left( n \right) \right] - \E \left[ M_{\utau}^{\left( T \right)} \left( n \right) \right] \right)^2}{2 \Var \left( M_{\utau}^{\left( S \right)} \left( n \right) \right) + 2 \Var \left( M_{\utau}^{\left( T \right)} \left( n \right) \right) +\left( \E \left[ M_{\utau}^{\left( S \right)} \left( n \right) \right] - \E \left[ M_{\utau}^{\left( T \right)} \left( n \right) \right] \right)^2}.
\end{multline*}
Since $M_{\utau}^{\left( S \right)}$ and $M_{\utau}^{\left( T \right)}$ are martingales, for every $n \geq n_0$ we have that $\E \left[ M_{\utau}^{\left( S \right)} \left( n \right) \right] - \E \left[ M_{\utau}^{\left( T \right)} \left( n \right) \right] = \E \left[ M_{\utau}^{\left( S \right)} \left( n_0 \right) \right] - \E \left[ M_{\utau}^{\left( T \right)} \left( n_0 \right) \right] \neq 0$. Also, since the two martingales are bounded in $L^2$, we have that $\Var \left( M_{\utau}^{\left( S \right)} \left( n \right) \right) + \Var \left( M_{\utau}^{\left( T \right)} \left( n \right) \right)$ is bounded as $n \to \infty$. We conclude that $\delta \left( S, T \right) > 0$. 
\end{proof}

\subsection{Recurrence relation}\label{sec:recurrence}

The following recurrence relation for the conditional expectations of $F_{\utau} \left( \UA \left( n, S \right) \right)$ is key to estimating the moments of $F_{\utau} \left( \UA \left( n, S \right) \right)$. 

\begin{lemma}\label{lem:recurrence}
 Let $\utau \in \cD$ such that $\left| \utau \right| \geq 2$. Then for every seed tree $S$ and for every $n \geq \left| S \right|$ we have 
\begin{multline}\label{eq:rec}
 \E \left[ F_{\utau} \left( \UA \left( n + 1, S \right) \right) \, \middle| \, \cF_n \right] = \left( 1 + \frac{w \left( \utau \right)}{n} \right) F_{\utau} \left( \UA \left( n, S \right) \right) \\
 + \frac{1}{n} \left\{ \sum_{v \in \tau : \ell \left( v \right) \geq 2} \ell \left( v \right) \left( \ell \left( v \right) - 1 \right) F_{\utau_v'} \left( \UA \left( n, S \right) \right) + \sum_{v \in L_{0,1} \left( \utau \right) } F_{\utau_v} \left( \UA \left( n, S \right) \right) \right\}.
\end{multline}
\end{lemma}

\begin{proof}
 Fix $\utau \in \cD$ with $\left| \utau \right| \geq 2$, fix a seed tree $S$, and let $n \geq \left| S \right|$. To simplify notation we omit the dependence on $S$ and write $T_n$ instead of $\UA \left( n, S \right)$. When evaluating $\E \left[ F_{\utau} \left( T_{n+1} \right) \, \middle| \, \cF_n \right]$ we work conditionally on $\cF_n$, so we may consider $T_n$ as being fixed. 

 \medskip

 Let $u_{n+1}$ denote the vertex present in $T_{n+1}$ but not in $T_n$, and let $u_n$ denote its neighbor in $T_{n+1}$. 
Let $\cE_{n+1}$ denote the set of all embeddings $\phi : \tau \to T_{n+1}$; we can write $\cE_{n+1}$ as the disjoint union of the set of those using only vertices of $T_n$, denoted by $\cE_n$, and the set of those using the new vertex $u_{n+1}$, denoted by $\cE_{n+1} \setminus \cE_n$. 
To simplify notation, if $\utau \in \cD$, $T$ is a tree, and $\phi : \tau \to T$ is an embedding, write $\cW_{\phi} \left( T \right) = \prod_{v \in \tau} \left[ f_{\phi \left( v \right)} \left( T \right) \right]_{\ell \left( v \right)}$ for the number of decorated embeddings of $\utau$ in $T$ that use the embedding $\phi$. 
We then have $F_{\utau} \left( T_{n+1} \right) = \sum_{\phi \in \cE_n} \cW_{\phi} \left( T_{n+1} \right) + \sum_{\phi \in \cE_{n+1} \setminus \cE_n} \cW_{\phi} \left( T_{n+1} \right)$ and we deal with the two sums separately. 

\medskip 

First let $\phi \in \cE_n$. For $v \in \tau$, denote by $E_v$ the event that $u_n$ is in the same tree of $\wh{T_n}$ as $\phi \left( v \right)$ (recall the definition of $\wh{T_n}$ from Section~\ref{sec:obs}). Clearly $\P \left( E_v \, \middle| \, \cF_n \right) = f_{\phi \left( v \right)} \left( T_n \right) / n$. 
Under the event $E_v$ we have that $f_{\phi\left( v \right)} \left( T_{n+1} \right) = f_{\phi\left( v \right)} \left( T_n \right) + 1$, while for every $v' \in \tau \setminus \left\{ v \right\}$ we have $f_{\phi\left( v' \right)} \left( T_{n+1} \right) = f_{\phi\left( v' \right)} \left( T_n \right)$. 
Now using the identities $\left[ d + 1 \right]_{\ell} = \left[ d \right]_{\ell} + \ell \times \left[ d \right]_{\ell - 1}$ and 
$d \times \left[ d \right]_{\ell - 1} = \left[ d \right]_{\ell} + \left( \ell - 1 \right) \times \left[ d \right]_{\ell - 1}$, 
which hold for every $d, \ell \geq 1$, and also using $\left[d+1 \right]_0 = \left[d\right]_0$, we have that
\begin{multline*}
 \E \left[ \cW_{\phi} \left( T_{n+1} \right) \, \middle| \, \cF_n \right] \\
\begin{aligned}
&= \sum_{v \in \tau} \frac{f_{\phi\left( v \right)} \left( T_n \right)}{n} \left[ f_{\phi \left( v \right)} \left( T_n \right) + 1 \right]_{\ell \left( v \right)} \prod_{v' \in \tau \setminus \left\{ v \right\}} \left[ f_{\phi \left( v' \right)} \left( T_n \right) \right]_{\ell \left( v' \right)} \\
&= \cW_{\phi} \left( T_n \right) + \frac{1}{n} \sum_{v \in \tau : \ell \left( v \right) \geq 1} \ell \left( v \right) f_{\phi \left( v \right)} \left( T_n \right) \left[ f_{\phi\left( v \right)} \left( T_n \right) \right]_{\ell \left( v \right) - 1} \prod_{v' \in \tau \setminus \left\{ v \right\}} \left[ f_{\phi \left( v' \right)} \left( T_n \right) \right]_{\ell \left( v' \right)} \\
&= \left( 1 + \frac{w \left( \utau \right)}{n} \right) \cW_{\phi} \left( T_n \right) + \frac{1}{n} \sum_{v \in \tau : \ell \left( v \right) \geq 2} \ell \left( v \right) \left( \ell \left( v \right) - 1 \right) \left[ f_{\phi\left( v \right)} \left( T_n \right) \right]_{\ell \left( v \right) - 1} \prod_{v' \in \tau \setminus \left\{ v \right\}} \left[ f_{\phi \left( v' \right)} \left( T_n \right) \right]_{\ell \left( v' \right)} \\
&= \left( 1 + \frac{w \left( \utau \right)}{n} \right) \cW_{\phi} \left( T_n \right) + \frac{1}{n} \sum_{v \in \tau : \ell \left( v \right) \geq 2} \ell \left( v \right) \left( \ell \left( v \right) - 1 \right) \cW_{\phi_v'} \left( T_n \right),
\end{aligned}
\end{multline*}
where $\phi_v'$ is the embedding equal to $\phi$ of the decorated tree $\utau_v'$. 
Now as $\phi$ runs through the embeddings of $\utau$ in $T_n$, $\phi_v'$ runs exactly through the embeddings of $\utau_v'$. 
So we have that
\begin{equation}\label{eq:phi_old}
  \sum_{\phi \in \cE_n} \E \left[ \cW_{\phi} \left( T_{n+1} \right) \, \middle| \, \cF_n \right] = \left( 1 + \frac{w \left( \utau \right)}{n} \right) F_{\utau} \left( T_n \right) + \frac{1}{n} \sum_{v \in \tau : \ell \left( v \right) \geq 2} \ell \left( v \right) \left( \ell \left( v \right) - 1  \right) F_{\utau_v'} \left( T_n \right).
\end{equation}

\medskip

Now fix $T_{n+1}$ and consider $\phi \in \cE_{n+1} \setminus \cE_n$. 
Let $w \in \tau$ be such that $\phi \left( w \right) = u_{n+1}$. 
Since $\phi$ is an embedding, we must have $w \in L \left( \utau \right)$. 
Note that if $w \notin L_{0,1} \left( \utau \right)$ then $\cW_{\phi} \left( T_{n+1} \right) = 0$. 
If $w \in L_{0,1} \left( \utau \right)$ then denote by $\cE_w$ the set of all embeddings $\phi \in \cE_{n+1} \setminus \cE_n$ such that $\phi \left( w \right) = u_{n+1}$. 
Now fix $w \in L_{0,1} \left( \utau \right)$ and $\phi \in \cE_w$. 
Note that $\phi$ restricted to $\tau \setminus \left\{ w \right\}$ is an embedding of $\utau_w$ in $T_n$; call this $\phi_w$. 
Let $x$ be the neighbor of $w$ in $\tau$. 
We then must have $\phi \left( x \right) = u_n$, and also $\left[ f_{\phi \left( w \right)} \left( T_{n+1} \right) \right]_{\ell \left( w \right)} = 1$ (irrespective of whether $w \in L_0 \left( \utau \right)$ or $w \in L_1 \left( \utau \right)$).  
Furthermore, for every $w' \in \tau \setminus \left\{ w \right\}$ we have $f_{\phi \left( w' \right)} \left( T_{n+1} \right) = f_{\phi \left( w' \right) } \left( T_n \right)$. 
Thus we have 
\[
 \cW_{\phi} \left( T_{n+1} \right) = \cW_{\phi_w} \left( T_n \right) \mathbf{1}_{\left\{ \phi \left( w \right) = u_{n+1} \right\}}.
\]
For fixed $w \in L_{0,1} \left( \utau \right)$, as $\phi$ runs through $\cE_w$, $\phi_w$ runs through all the embeddings of $\utau_w$ in $T_n$. So summing over $w \in L$ we obtain 
\[
 \sum_{\phi \in \cE_{n+1} \setminus \cE_n} \cW_{\phi} \left( T_{n+1} \right) = \sum_{w \in L_{0,1} \left( \utau \right)} \sum_{\phi_w : \utau_w \to T_n} \cW_{\phi_w} \left( T_n \right) \mathbf{1}_{\left\{ \phi \left( w \right) = u_{n+1} \right\}} =  \sum_{w \in L_{0,1} \left( \utau \right)} F_{\utau_w} \left( T_n \right) \mathbf{1}_{\left\{ \phi \left( w \right) = u_{n+1} \right\}}.
\]
Now taking conditional expectation given $\mathcal{F}_n$, we get that
\begin{equation}\label{eq:phi_new}
 \E \left[ \sum_{\phi \in \cE_{n+1} \setminus \cE_n} \cW_{\phi} \left( T_{n+1} \right) \, \middle| \, \cF_n \right] = \frac{1}{n} \sum_{w \in L_{0,1} \left( \utau \right)} F_{\utau_w} \left( T_n \right).
\end{equation}
Summing~\eqref{eq:phi_old} and~\eqref{eq:phi_new} we obtain~\eqref{eq:rec}.
\end{proof}

\subsection{Moment estimates}\label{sec:moment}

Using the recurrence relation of Lemma~\ref{lem:recurrence} proved in the previous subsection, we now establish moment estimates on the number of decorated embeddings $F_{\utau} \left( \UA \left( n, S \right) \right)$. 
These are then used in the next subsection to show that the martingales of Proposition~\ref{prop:mg} are bounded in $L^2$. 

\medskip

The first moment estimates are a direct corollary of Lemma~\ref{lem:recurrence}.
\begin{corollary}\label{cor:first_moment}
 Let $\utau \in \cD$ be a decorated tree and let $S$ be a seed tree. 
\begin{enumerate}[(a)]
 \item We have that $n^{w\left( \utau\right)} \simless \E \left[ F_{\utau} \left( \UA \left( n, S \right) \right) \right]$.
 \item If $\left| \utau \right| \geq 2$ and $\utau \in \cD \setminus \cD_0^*$, then we have that $\E \left[ F_{\utau} \left( \UA \left( n, S \right) \right) \right] \simless n^{w\left( \utau\right)}$.
 \item If $\left| \utau \right| = 1$ or $\utau \in \cD_0^*$, then we have that $\E \left[ F_{\utau} \left( \UA \left( n, S \right) \right) \right] \simless n^{w\left( \utau\right)+1}$.
\end{enumerate}
\end{corollary}
\begin{proof} 
Fix a seed tree $S$ and, as before, write $T_n$ instead of $\UA \left( n, S \right)$ in order to simplify notation. 
First, recall that if $\left| \utau \right| = 1$, then $F_{\utau} \left( T_n \right) = n \times \left[ n \right]_{w \left( \utau \right)}$, so the statements of part (a) and (c) hold in this case. In the following we can therefore assume that $\left| \utau \right| \geq 2$.

\medskip

Lemma~\ref{lem:recurrence} then implies that for every $n \geq \left| S \right|$ we have
\[
 \E \left[ F_{\utau} \left( T_{n+1} \right) \right] \geq \left( 1 + \frac{w\left( \utau \right)}{n} \right) \E \left[ F_{\utau} \left( T_n \right) \right].
\]
Since there exists $n_0$ such that $\E \left[ F_{\utau} \left( T_{n_0} \right) \right] > 0$ (one can take, e.g., $n_0 = \left| S \right| + \left| \utau \right|$), this immediately implies part (a) (see Appendix~\ref{app:est} for further details). 

\medskip

We prove parts (b) and (c) by induction on $\utau$ for the partial order $\preccurlyeq$. We have already checked that the statement holds when $\left| \utau \right| = 1$, so the base case of the induction holds. 
Now fix $\utau$ such that $\left| \utau \right| \geq 2$, and assume that (b) and (c) hold for all $\utau'$ such that $\utau' \prec \utau$. There are two cases to consider: either $\utau \in \cD \setminus \cD_0^*$ or $\utau \in \cD_0^*$. 

\medskip

If $\utau \in \cD \setminus \cD_0^*$ then $L_0 \left( \utau \right) = \emptyset$. Note that for every $v \in L_1 \left( \utau \right)$ we have $w \left( \utau_v \right) = w \left( \utau \right) - 1$, and also for every $v \in \tau$ such that $\ell \left( v \right) \geq 2$ we have $w\left( \utau_v' \right) = w \left( \utau \right) - 1$. 
Therefore by induction for every $v \in L_1 \left( \utau \right)$ we have $\E \left[ F_{\utau_v} \left( T_n \right) \right] \simless n^{w \left( \utau \right)}$ and also for every $v \in \tau$ such that $\ell \left( v \right) \geq 2$ we have $\E \left[ F_{\utau_v'} \left( T_n \right) \right] \simless n^{w \left( \utau \right)}$. 
So by Lemma~\ref{lem:recurrence} there exist constants $C, \gamma > 0$ such that
\begin{equation}\label{eq:ub}
 \E \left[ F_{\utau} \left( T_{n+1} \right) \right] \leq \left( 1 + \frac{w\left( \utau \right)}{n} \right) \E \left[ F_{\utau} \left( T_n \right) \right] + C \left( \log \left( n \right) \right)^{\gamma} n^{w \left( \utau \right) - 1}.
\end{equation}
This then implies that $\E \left[ F_{\utau} \left( T_n \right) \right] \simless n^{w \left( \utau\right)}$; see Appendix~\ref{app:est} for details.

\medskip

If $\utau \in \cD_0^*$ then $L_0 \left( \utau \right) \neq \emptyset$ and note that for every $v \in L_0 \left( \utau \right)$ we have $w \left( \utau_v \right) = w \left( \utau \right)$. 
If for every $v \in L_0 \left( \utau \right)$ we have $\utau_v \in \cD \setminus \cD_0^*$, then the same argument as in the previous paragraph goes through, and we have that $\E \left[ F_{\utau} \left( T_n \right) \right] \simless n^{w \left( \utau\right)}$. 
However, if there exists $v \in L_0 \left( \utau \right)$ such that $\utau_v \in \cD_0^*$, then~\eqref{eq:ub} does not hold in this case; instead we have from Lemma~\ref{lem:recurrence} that there exist constants $C, \gamma > 0$ such that 
\[
 \E \left[ F_{\utau} \left( T_{n+1} \right) \right] \leq \left( 1 + \frac{w\left( \utau \right)}{n} \right) \E \left[ F_{\utau} \left( T_n \right) \right] + C \left( \log \left( n \right) \right)^{\gamma} n^{w \left( \utau \right)}.
\]
Similarly as before, this then implies that $\E \left[ F_{\utau} \left( T_n \right) \right] \simless n^{w \left( \utau\right)+1}$; see Appendix~\ref{app:est} for details.
\end{proof}

The second moment estimates require additional work. 
First, recall again that if $\left| \utau \right| = 1$, then $F_{\utau} \left( \UA \left( n, S \right) \right) = n \times \left[ n \right]_{w \left( \utau \right)}$. 
Consequently, we have that $F_{\utau} \left( \UA \left(n, S \right) \right)^2 \simless n^{2 w \left( \utau \right) + 2}$ and also that $\left( F_{\utau} \left( \UA \left( n + 1, S \right) \right) - F_{\utau} \left( \UA \left( n, S \right) \right) \right)^2 \simless n^{2 w \left( \utau \right)}$.

\begin{lemma}\label{lem:second_moment}
 Let $\utau \in \cD_+$ with $\left| \utau \right| \geq 2$ and let $S$ be a seed tree. 
\begin{enumerate}[(a)]
 \item We have that $\E \left[ F_{\utau} \left( \UA \left(n, S \right) \right)^2 \right] \simless n^{2 w \left( \utau \right)}$.
 \item We have that $\E \left[ \left( F_{\utau} \left( \UA \left( n + 1, S \right) \right) - F_{\utau} \left( \UA \left( n, S \right) \right) \right)^2 \right] \simless n^{2 w \left( \utau \right) - 2}$.
\end{enumerate}
\end{lemma}
We note that part (a) of the lemma follows in a short and simple way once part (b) is proven. 
However, for expository reasons, we first prove part (a) directly, and then prove part (b), whose proof is similar to, and builds upon, the proof of part (a). 

\medskip 

\begin{proof}
 Fix $\utau \in \cD_+$ with $\left| \utau \right| \geq 2$ and a seed tree $S$. 
Define $K \equiv K \left( \utau \right) := \max\left\{ 4 \left( \left| \utau \right| + w \left( \utau \right) \right), 20 \right\}$. 
We always have $F_{\utau} \left( \UA \left(n, S \right) \right) \leq n^{K/4}$, 
since the number of embeddings of $\tau$ in $\UA \left( n, S \right)$ is at most $n^{\left| \utau \right|}$, 
and the product of the subtree sizes raised to appropriate powers is at most $n^{w\left( \utau \right)}$. 
By Lemma~\ref{lem:diam} in Appendix~\ref{app:diam} there exists a constant $C\left( S \right)$ such that $\P \left( \diam \left( \UA \left( n, S \right) \right) > K \log \left( n \right) \right) \leq C\left( S \right) n^{-K/2}$. 
Therefore we have
\[
 \E \left[ F_{\utau} \left( \UA \left(n, S \right) \right)^2 \mathbf{1}_{\left\{ \diam \left( \UA \left( n, S \right) \right) > K \log \left( n \right)  \right\}} \right] \leq C\left( S \right),
\]
and similarly
\[
 \E \left[ \left( F_{\utau} \left( \UA \left( n + 1, S \right) \right) - F_{\utau} \left( \UA \left( n, S \right) \right) \right)^2 \mathbf{1}_{\left\{ \diam \left( \UA \left( n, S \right) \right) > K \log \left( n \right)  \right\}} \right] \simless 1.
\]
Therefore in the remainder of the proof we may, roughly speaking, assume that $\diam \left( \UA \left( n, S \right) \right) \leq K \log \left( n \right)$; 
this will be made precise later.

\medskip

To simplify notation, write simply $T_n$ instead of $\UA \left( n, S \right)$. 
Our proof is combinatorial and uses the notion of decorated embeddings as described in Section~\ref{sec:obs}. 
We start with the proof of (a) which is simpler. 
% Let $\utau'$ be a disjoint copy of $\utau$ and denote by $\utau \sqcup \utau'$ the disjoint union of the two trees. 
% We say that $\uphi : \utau \sqcup \utau' \to T_n$ is a \emph{decorated map} if it is a map such that both $\uphi|_{\utau}$ and $\uphi|_{\utau'}$ are decorated embeddings. 
We say that $\uphi = \uphi_1 \times \uphi_2$ is a \emph{decorated map} if it is a map such that both $\uphi_1$ and $\uphi_2$ are decorated embeddings from $\utau$ to $T_n$. 
Note that $\uphi$ is not necessarily injective. 
If $\uphi$ is a decorated embedding or a decorated map, we denote by $\phi$ the map of the tree without the choices of vertices associated with the arrows. 

\medskip

Now observe that $F_{\utau} \left( T_n \right)^2$ is exactly the number of decorated maps $\uphi = \uphi_1 \times \uphi_2$. 
We partition the set of such decorated maps into two parts: 
let $\cE_{\utau}^1 \left( T_n \right)$ denote the set of all such decorated maps where $\phi_1 \left( \tau \right) \cap \phi_2 \left( \tau \right) \neq \emptyset$, and
let $\cE_{\utau}^2 \left( T_n \right)$ denote the set of all such decorated maps where $\phi_1 \left( \tau \right) \cap \phi_2 \left( \tau \right) = \emptyset$. 
Clearly $F_{\utau} \left( T_n \right)^2 = \left| \cE_{\utau}^1 \left( T_n \right) \right| + \left| \cE_{\utau}^2 \left( T_n \right) \right|$. 
This partition is not necessary for the proof, but it is helpful for exposition. 

\medskip

We first estimate $\left| \cE_{\utau}^1 \left( T_n \right) \right|$. 
To do this, we associate to each decorated map $\uphi \in \cE_{\utau}^1 \left( T_n \right)$ a decorated tree $\usigma$ and a decorated embedding $\upsi$ of it in $T_n$, in the following way; 
see also Figure~\ref{fig:dec_map} for an illustration. 
We take simply the union of the images of the decorated embeddings $\uphi_1$ and $\uphi_2$, and if these share any vertices, edges, or arrows, then we identify them (i.e., we take only a single copy). 
The resulting union is the image of a decorated tree $\usigma$ under a decorated embedding $\upsi$; 
note that $\usigma$ is uniquely defined, and $\upsi$ is uniquely defined up to the ordering of the arrows associated with $\usigma$. 
To define $\upsi$ uniquely, we arbitrarily define the ordering of the arrows associated with $\usigma$ to be the concatenation of the orderings associated with $\uphi_1$ and $\uphi_2$. 
Here we used the fact that $\uphi \in \cE_{\utau}^1 \left( T_n \right)$, since when $\phi_1 \left( \tau \right) \cap \phi_2 \left( \tau \right) = \emptyset$, the union of the two decorated embeddings cannot be the image of a single decorated tree under a decorated embedding.

\begin{figure}[h!]
  \centering
    \includegraphics[width=\textwidth]{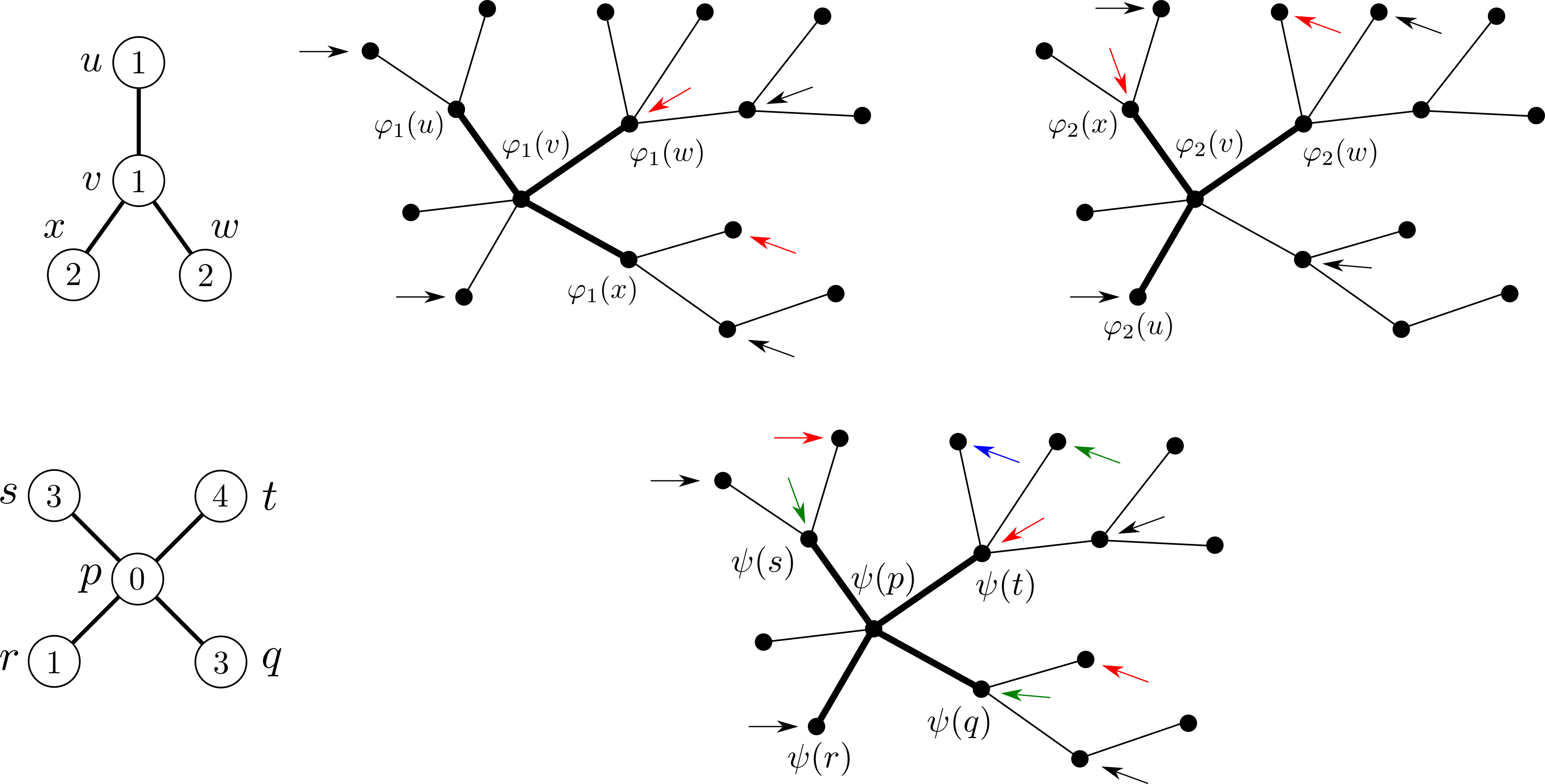}
  \caption{\textbf{A decorated map and an associated decorated embedding.} The top row depicts a decorated tree $\utau$ and two decorated embeddings, $\uphi_1$ and $\uphi_2$, of it into a larger tree $T$. The bottom row depicts the associated decorated tree $\usigma$, and the decorated embedding $\upsi$ of it into $T$.}
  \label{fig:dec_map}
\end{figure}

%\medskip

Note that 
% $\utau$ and $\utau'$ both have $w \left( \utau \right)$ arrows associated with them, and 
when taking the union of the decorated embeddings we do not introduce any new arrows, so we must have $w\left( \usigma \right) \leq 2 w\left( \utau \right)$. 
Note also that, due to the nonlocality of the decorations, $\usigma$ might have vertices having a label being zero, see, e.g., Figure~\ref{fig:dec_map}. 
However, importantly, the construction implies that all leaves of $\usigma$ have positive labels, i.e., $\usigma \in \cD \setminus \cD_0^*$. 

\medskip

Let $\cU \left( \utau \right)$ denote the set of all decorated trees $\usigma$ that can be obtained in this way. 
The cardinality of $\cU \left( \utau \right)$ is bounded above by a constant depending only on $\utau$, as we now argue. 
The number of ways that two copies of $\tau$ can be overlapped clearly depends only on $\tau$. 
Once the union $\sigma$ of the two copies of $\tau$ is fixed, only the arrows need to be associated with vertices of $\sigma$. 
There are at most $2w \left( \utau \right)$ arrows, and $\sigma$ has at most $2 \left| \utau \right|$ vertices, so there are at most $\left( 2 \left| \utau \right| \right)^{2 w \left( \utau \right)}$ ways to associate arrows to vertices. 

\medskip 

The function $\uphi \mapsto \left( \usigma, \upsi \right)$ is not necessarily one-to-one. 
However, there exists a constant $c \left( \utau \right)$ depending only on $\utau$ such that any pair $\left( \usigma, \upsi \right)$ is associated with at most $c \left( \utau \right)$ decorated maps $\uphi$. 
To see this, note that given $\left( \usigma, \upsi \right)$, in order to recover $\uphi$, it is sufficient to know the following:
(i) for every edge of $\psi \left( \sigma \right)$, whether it is a part of $\phi_1 \left( \tau \right)$, a part of $\phi_2 \left( \tau \right)$, or a part of both, 
(ii) for every arrow of $\upsi \left( \usigma \right)$, whether it is a part of $\uphi_1 \left( \utau \right)$, a part of $\uphi_2 \left( \utau \right)$, or a part of both, and
(iii) the ordering of the arrows for $\uphi_1$ and $\uphi_2$. 
Since $\left| \usigma \right| \leq 2 \left| \utau \right|$ and $w\left( \usigma \right) \leq 2 w \left( \utau \right)$, we can take $c \left( \utau \right) = 3^{2 \left| \utau \right| + 2 w \left( \utau \right)} \left( w\left( \utau \right)! \right)^2$. 

\medskip

We have thus shown that 
\[
 \left| \cE_{\utau}^1 \left( T_n \right) \right| \leq c \left( \utau \right) \sum_{\usigma \in \cU \left( \utau \right)} F_{\usigma} \left( T_n \right).
\]
For every $\usigma \in \cU \left( \utau \right)$ we have that $\usigma \in \cD \setminus \cD_0^*$, $\left| \usigma \right| \geq \left| \utau \right| \geq 2$, and $w \left( \usigma \right) \leq 2 w \left( \utau \right)$, and so by Corollary~\ref{cor:first_moment} we have that $\E \left[ F_{\usigma} \left( T_n \right) \right] \simless n^{2 w \left( \utau \right)}$. 
Since the cardinality of $\cU \left( \utau \right)$ depends only on $\utau$, this implies that $\E \left[ \left| \cE_{\utau}^1 \left( T_n \right) \right| \right] \simless n^{2 w \left( \utau \right)}$. 

\medskip

Now we turn to estimating $\left| \cE_{\utau}^2 \left( T_n \right) \right|$. 
Again, we associate to each decorated map $\uphi \in \cE_{\utau}^2 \left( T_n \right)$ a decorated tree $\usigma$ and a decorated embedding $\upsi$ of it in $T_n$. 
This is done by first, just as before, taking the union of the images of the decorated embeddings $\uphi_1$ and $\uphi_2$, and if these share any arrows, identifying them (i.e., we take only a single copy). 
(Note that, since $\phi_1 \left( \tau \right) \cap \phi_2 \left( \tau \right) = \emptyset$, the decorated embeddings do not share any vertices or edges; however, due to the nonlocality of decorations they might share arrows.) 
We then take the union of this with the unique path in $T_n$ that connects $\phi_1 \left( \tau \right)$ and $\phi_2 \left( \tau \right)$. 
The result of this is a tree in $T_n$, together with a set of at most $2 w\left( \utau \right)$ arrows associated with vertices of $T_n$; 
this is thus the image of a decorated tree $\usigma$ under a decorated embedding $\upsi$, and this is how we define $\left( \usigma, \upsi \right)$. 
See Figure~\ref{fig:dec_map2} for an illustration. 

\begin{figure}[h!]
  \centering
    \includegraphics[width=\textwidth]{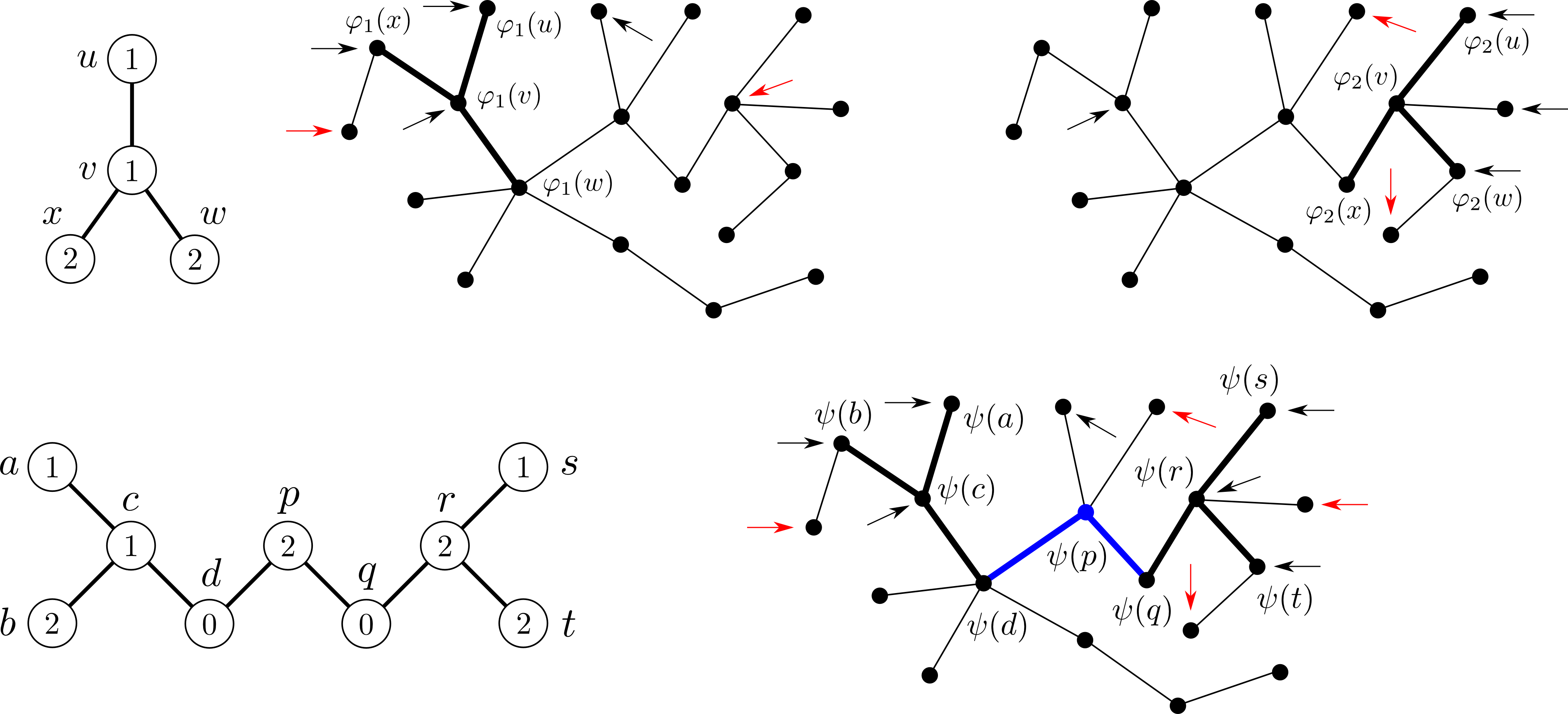}%{dec_tree_map2_color}
  \caption{\textbf{Another decorated map and an associated decorated embedding.} The top row depicts a decorated tree $\utau$ and two decorated embeddings, $\uphi_1$ and $\uphi_2$, of it into a larger tree $T$, where now $\phi_1 \left( \tau \right) \cap \phi_2 \left( \tau \right) = \emptyset$. 
The bottom row depicts the associated decorated tree $\usigma$, and the decorated embedding $\upsi$ of it into $T$. The path connecting $\phi_1 \left( \tau \right)$ and $\phi_2 \left( \tau \right)$ is depicted in blue.}
  \label{fig:dec_map2}
\end{figure}

\medskip

Again we have that any such $\usigma$ must satisfy $w \left( \usigma \right) \leq 2 w \left( \utau \right)$ and $\usigma \in \cD \setminus \cD_0^*$. 
The important difference now is that a priori we have no bound on $\left| \usigma \right|$. 
This is where we use that $\diam \left( T_n \right) \leq K \log \left( n \right)$ with high probability.
%(which we can assume, as argued at the beginning of the proof).

\medskip

Let $\cU_2^{\left( n \right)} \left( \utau \right)$ denote the set of all decorated trees $\usigma$ of diameter at most $K \log \left( n \right)$ that can be obtained in this way. 
The cardinality of $\cU_2^{\left( n \right)} \left( \utau \right)$ cannot be bounded above by a constant depending only on $\utau$, but
it is at most polylogarithmic in $n$, as we now argue. 
There are at most $\left| \utau \right|^2$ ways to choose which vertices of $\phi_1 \left( \tau \right)$ and $\phi_2 \left( \tau \right)$ are closest to each other, and the path connecting them has length at most $K \log \left( n \right)$. So the number of trees $\sigma$ that can be obtained is at most $\left| \utau \right|^2 K \log \left( n \right)$. 
Once the tree $\sigma$ is fixed, only the arrows need to be associated with vertices of $\sigma$. 
There are at most $2 w \left( \utau \right)$ arrows, and $\sigma$ has at most $K \log \left( n \right) + 2 \left| \utau \right|$ vertices, which shows that
\[
 \left| \cU_2^{\left( n \right)} \left( \utau \right) \right| \leq \left| \utau \right|^2 K \log \left( n \right) \left( K \log \left( n \right) + 2 \left| \utau \right| \right)^{2w\left( \utau \right)} \simless 1.
\]

\medskip

The function $\uphi \mapsto \left( \usigma, \upsi \right)$ is not one-to-one. 
However, there exists a constant $c_2\left( \utau \right)$ depending only on $\utau$ such that 
any pair $\left( \usigma, \upsi \right)$ is associated with at most $c_2 \left( \utau \right)$ decorated maps $\uphi$, as we now show. 
First, given $\left( \usigma, \upsi \right)$, we know that $\phi_1 \left( \tau \right)$ and $\phi_2 \left( \tau \right)$ are at the two ``ends'' of $\psi \left( \sigma \right)$. 
The two ``ends'' of $\psi \left( \sigma \right)$ are well-defined: an edge $e$ of $\psi \left( \sigma \right)$ is part of the path connecting $\phi_1 \left( \tau \right)$ and $\phi_2 \left( \tau \right)$ (and hence not part of an ``end'') if and only if there are at least $\left| \utau \right|$ vertices on both sides of the cut defined by $e$. 
In order to recover $\uphi$, we also need to know for each arrow of $\upsi \left( \usigma \right)$, whether it is a part of $\uphi_1 \left( \utau \right)$, a part of $\uphi_2 \left( \utau \right)$, or a part of both. 
Finally, we need to know the ordering of the arrows for $\uphi_1$ and $\uphi_2$. 
Since $w \left( \usigma \right) \leq 2 w \left( \utau \right)$, we can thus take $c_2 \left( \utau \right) = 2 \times 3^{2 w\left( \utau \right)} \left( w \left( \utau \right)! \right)^2$. 

\medskip

We have thus shown that 
\[
 \left| \cE_{\utau}^2 \left( T_n \right) \right| \mathbf{1}_{\left\{ \diam \left( T_n \right) \leq K \log \left( n \right) \right\}}  \leq c_2 \left( \utau \right) \sum_{\usigma \in \cU_2^{\left( n \right)} \left( \utau \right)} F_{\usigma} \left( T_n \right).
\]
For every $\usigma \in \cU_2^{\left( n \right)} \left( \utau \right)$ we have that $\usigma \in \cD \setminus \cD_0^*$, $\left| \usigma \right| \geq \left| \utau \right| \geq 2$, and $w \left( \usigma \right) \leq 2 w \left( \utau \right)$, and so by Corollary~\ref{cor:first_moment} we have that $\E \left[ F_{\usigma} \left( T_n \right) \right] \simless n^{2 w \left( \utau \right)}$. 
Since we have $\left| \cU_2^{\left( n \right)} \left( \utau \right) \right| \simless 1$, we thus have
\[
 \E \left[ \left| \cE_{\utau}^2 \left( T_n \right) \right| \mathbf{1}_{\left\{ \diam \left( T_n \right) \leq K \log \left( n \right) \right\}} \right] \simless n^{2 w\left( \utau \right)}.
\]
This concludes the proof of (a). 

\medskip

For the proof of (b) we work conditionally on $\cF_n$. 
As in the proof of Lemma~\ref{lem:recurrence}, let $u_{n+1}$ denote the vertex present in $T_{n+1}$ but not in $T_n$, and let $u_n$ denote its neighbor in $T_{n+1}$. 
Observe that $F_{\utau} \left( T_{n+1} \right) - F_{\utau} \left(  T_n \right)$ is equal to the number of decorated embeddings of $\utau$ in $T_{n+1}$ that use the new vertex $u_{n+1}$. 
There are two ways that this may happen, and we call such decorated embeddings ``type A'' and ``type B'' accordingly (see Figure~\ref{fig:dec_mapAB} for an illustration): 
\begin{itemize}
 \item \textbf{Type A.} The decorated embedding maps a vertex $v \in \tau$ to $u_{n+1}$. Since $\utau \in \cD_+$ and the arrows are mapped to different vertices, we must then have $\ell \left( v \right) = 1$, and the arrow pointing to $v$ in $\utau$ must be mapped to $u_{n+1}$. 
 \item \textbf{Type B.} The decorated embedding maps $\tau$ in $T_n$, but there exists an arrow of $\utau$ which it maps to $u_{n+1}$. 
\end{itemize}
\begin{figure}[h!]
  \centering
    \includegraphics[width=\textwidth]{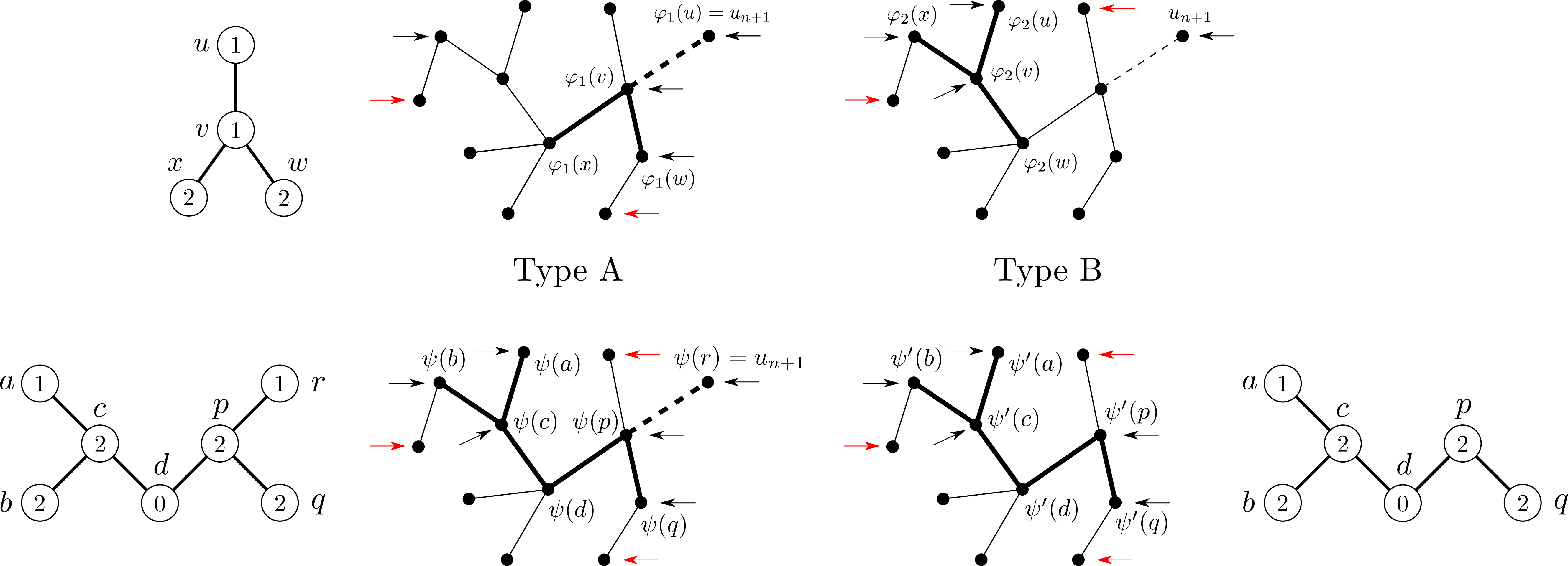}
  \caption{\textbf{Type A and type B decorated embeddings.} The top row depicts a decorated tree $\utau$ and two decorated embeddings, $\uphi_1$ and $\uphi_2$, of it into a larger tree $T_{n+1}$. Here $\uphi_1$ is of type A, and $\uphi_2$ is of type B. In the bottom left is the associated decorated tree $\usigma$, together with the decorated embedding $\upsi$ of it into $T_{n+1}$. In the bottom right is the pair $\left( \usigma', \upsi' \right)$.}
  \label{fig:dec_mapAB}
\end{figure}
Consequently $\left( F_{\utau} \left( T_{n+1} \right) - F_{\utau} \left(  T_n \right) \right)^2$ is equal to the number of decorated maps 
% $\uphi : \utau \sqcup \utau' \to T_{n+1}$ such that 
$\uphi = \uphi_1 \times \uphi_2$ such that 
% both $\uphi|_{\utau}$ and $\uphi|_{\utau'}$ are either of type A or of type B. 
$\uphi_1$ is either of type A or of type B, and the same holds for $\uphi_2$. 
We denote by $\tilde{\cE}_{\utau}^1 \left( T_{n+1} \right)$ the set of all such decorated maps where $\phi_1 \left( \tau \right) \cap \phi_2 \left( \tau \right) \neq \emptyset$, 
and let $\tilde{\cE}_{\utau}^2 \left( T_{n+1} \right)$ denote the set of all such decorated maps where $\phi_1 \left( \tau \right) \cap \phi_2 \left( \tau \right) = \emptyset$. 
Thus we have $\left( F_{\utau} \left( T_{n+1} \right) - F_{\utau} \left(  T_n \right) \right)^2 = \left| \tilde{\cE}_{\utau}^1 \left( T_{n+1} \right) \right| + \left| \tilde{\cE}_{\utau}^2 \left( T_{n+1} \right) \right|$. 
Again, this partition is not necessary for the proof, but it helps the exposition. 

\medskip

We first estimate $\left| \tilde{\cE}_{\utau}^1 \left( T_{n+1} \right) \right|$. 
In the same way as in part (a), we associate to each decorated map $\uphi \in \tilde{\cE}_{\utau}^1 \left( T_{n+1} \right)$ a pair $\left( \usigma, \upsi \right)$. 
Note that both $\uphi_1$ and $\uphi_2$ map an arrow to $u_{n+1}$, so $w \left( \usigma \right) \leq 2 w \left( \utau \right) - 1$, 
and also there exists an arrow $\ua^* \in \usigma$ that is mapped to $u_{n+1}$, denoted by $\upsi \left( \ua^* \right) = u_{n+1}$. 
We again have $\usigma \in \cD \setminus \cD_0^*$. 
As before, the set $\tilde{\cU} \left( \utau \right)$ of all decorated trees $\usigma$ that can be obtained in this way has cardinality bounded above by a constant depending only on $\utau$. 
Furthermore, there exists a constant $\tilde{c}\left( \utau \right)$ depending only on $\utau$ such that any pair $\left( \usigma, \upsi \right)$ is associated with at most $\tilde{c}\left( \utau \right)$ decorated maps $\uphi$. 

\medskip

We partition $\tilde{\cE}_{\utau}^1 \left( T_{n+1} \right)$ further into two parts. 
Let $\tilde{\cE}_{\utau}^{1,A} \left( T_{n+1} \right)$ denote the set of decorated maps $\uphi \in \tilde{\cE}_{\utau}^1 \left( T_{n+1} \right)$ such that at least one of $\uphi_1$ and $\uphi_2$ is of type $A$, 
and let $\tilde{\cE}_{\utau}^{1,B} \left( T_{n+1} \right) := \tilde{\cE}_{\utau}^1 \left( T_{n+1} \right) \setminus \tilde{\cE}_{\utau}^{1,A} \left( T_{n+1} \right)$. 
That is, $\tilde{\cE}_{\utau}^{1,B} \left( T_{n+1} \right)$ consists of those decorated maps $\uphi \in \tilde{\cE}_{\utau}^1 \left( T_{n+1} \right)$ such that both $\uphi_1$ and $\uphi_2$ is of type $B$. 

\medskip

We first estimate $\left| \tilde{\cE}_{\utau}^{1,A} \left( T_{n+1} \right) \right|$. 
We associate to each $\uphi \in \tilde{\cE}_{\utau}^{1,A} \left( T_{n+1} \right)$ a pair $\left( \usigma, \upsi \right)$ as above. 
Let $v \in \sigma$ denote the vertex such that $\psi \left( v \right) = u_{n+1}$, and 
let $v' \in \sigma$ denote the vertex such that $\psi \left( v' \right) = u_n$ 
(these vertices exist because $\uphi \in \tilde{\cE}_{\utau}^{1,A} \left( T_{n+1} \right)$). 
Define the decorated tree $\usigma'$ from $\usigma$ by removing the vertex $v$ from $\usigma$, as well as the arrow $\ua^*$ pointing to it. 
Define also the decorated embedding $\upsi' : \usigma' \to T_n$ to be equal to $\upsi$ on $\usigma'$, i.e., $\upsi' = \upsi|_{\usigma'}$; see Figure~\ref{fig:dec_mapAB} for an illustration. 
We have that $w \left( \usigma' \right) = w \left( \usigma \right) - 1 \leq 2 w \left( \utau \right) - 2$, 
it can be checked that $\usigma' \in \cD \setminus \cD_0^*$, 
and we also have $\psi' \left( v' \right) = u_n$. 
Let $\tilde{\cU}' \left( \utau \right)$ denote the set of all decorated trees $\usigma'$ that can be obtained in this way, and note that the cardinality of $\tilde{\cU}' \left( \utau \right)$ is bounded from above by a constant depending only on $\utau$. 
Since the map $\left( \usigma, \upsi \right) \mapsto \left( \usigma', \upsi', v' \right)$ is one-to-one, we have obtained that
\[
 \left| \tilde{\cE}_{\utau}^{1,A} \left( T_{n+1} \right) \right| \leq \sum_{\usigma' \in \tilde{\cU}' \left( \utau \right)} \sum_{v' \in \sigma'} \sum_{\upsi' : \usigma' \to T_n} \tilde{c} \left( \utau \right) \mathbf{1}_{\left\{ \psi' \left( v' \right) = u_n \right\}}.
\]
Since $u_n$ is uniform, we obtain
\[
 \E \left[ \left| \tilde{\cE}_{\utau}^{1,A} \left( T_{n+1} \right) \right| \right] = \E \left[ \E \left[ \left| \tilde{\cE}_{\utau}^{1,A} \left( T_{n+1} \right) \right| \, \middle| \, \cF_n \right] \right] \leq \sum_{\usigma' \in \tilde{\cU}' \left( \utau \right)} \sum_{v' \in \sigma'} \frac{\tilde{c} \left( \utau \right)}{n} \E \left[ F_{\usigma'} \left( T_n \right) \right] \simless n^{2w\left( \utau \right) - 3},
\]
where in the last inequality we used that for every $\usigma' \in \tilde{\cU}' \left( \utau \right)$ we have  $w \left( \usigma' \right) \leq 2 w \left( \utau \right) - 2$ and $\usigma' \in \cD \setminus \cD_0^*$, and so by Corollary~\ref{cor:first_moment} we have that $\E \left[ F_{\usigma'} \left( T_n \right) \right] \simless n^{2 w\left( \utau \right) - 2}$. 

\medskip

We now turn to estimating $\left| \tilde{\cE}_{\utau}^{1,B} \left( T_{n+1} \right) \right|$. 
Let 
\[
 v^* := \argmin_{v \in \sigma} \dist_{T_{n+1}} \left( \psi \left( v \right), u_{n+1} \right) = \argmin_{v \in \sigma} \dist_{T_n} \left( \psi \left( v \right), u_n \right),
\]
where $\dist_G$ denotes graph distance in a graph $G$. 
Note that the arrow $\ua^* \in \usigma$ is associated with $v^*$ in $\usigma$. 
Define the decorated map $\usigma^*$ from $\usigma$ by removing the arrow $\ua^*$ from $\usigma$. 
We have that $w \left( \usigma^* \right) = w \left( \usigma \right) - 1 \leq 2 w \left( \utau \right) - 2$. 
Furthermore, either $\usigma^* \in \cD \setminus \cD_0^*$ or $v^*$ is the only leaf of $\usigma$ that has label zero. 
Let $\tilde{\cU}^* \left( \utau \right)$ denote the set of all decorated trees $\usigma^*$ that can be obtained in this way, and note that the cardinality of $\tilde{\cU}^* \left( \utau \right)$ is bounded from above by a constant depending only on $\utau$. 
Define also the decorated embedding $\upsi^* : \usigma^* \to T_n$ to be equal to $\upsi$ on $\usigma^*$, i.e., $\upsi^* = \upsi|_{\usigma^*}$. 
Define furthermore $z^*$ to be the neighbor of $\upsi \left( \ua^* \right)$ in $T_{n+1}$; we thus have $z^* = u_n$. 
% Since the map $\left( \usigma, \upsi \right) \mapsto \left( \usigma^*, \upsi^*, z^* \right)$ is one-to-one, we have obtained that
Due to the ordering of the arrows, the map $\left( \usigma, \upsi \right) \mapsto \left( \usigma^*, \upsi^*, z^* \right)$ is not necessarily one-to-one, 
but any triple $\left( \usigma^*, \upsi^*, z^* \right)$ is associated with at most $w\left( \utau \right)$ pairs $\left( \usigma, \upsi \right)$. Thus, defining $\tilde{c}' \left( \utau \right) := \tilde{c} \left( \utau \right) w \left( \utau \right)$, we have that 
\begin{multline*}
 \left| \tilde{\cE}_{\utau}^{1,B} \left( T_{n+1} \right) \right| \\
\begin{aligned}
 & \leq \sum_{\usigma^* \in \tilde{\cU}^* \left( \utau \right)} \sum_{\upsi^* : \usigma^* \to T_n} \sum_{z^* \in T_n} \tilde{c}' \left( \utau \right) \mathbf{1}_{\left\{ z^* = u_n \right\}} \mathbf{1}_{\left\{ \usigma^* \in \cD \setminus \cD_0^* \right\} \cup \left\{ \usigma^* \in \cD_0^*, \argmin_{v \in \sigma^*} \dist_{T_n} \left( \psi \left( v \right), z^* \right) \in L_0 \left( \usigma^* \right) \right\}} \\
 & \leq \tilde{c}' \left( \utau \right) \sum_{\usigma^* \in \tilde{\cU}^* \left( \utau \right)} \sum_{\upsi^* : \usigma^* \to T_n}   \mathbf{1}_{\left\{ \usigma^* \in \cD \setminus \cD_0^* \right\} \cup \left\{ \usigma^* \in \cD_0^*, \argmin_{v \in \sigma^*} \dist_{T_n} \left( \psi \left( v \right), u_n \right) \in L_0 \left( \usigma^* \right) \right\}}.
\end{aligned}
\end{multline*}
Now if $\usigma^* \in \tilde{\cU}^* \left( \utau \right) \cap \left( \cD \setminus \cD_0^* \right)$, then the sum over embeddings $\upsi^* : \usigma^* \to T_n$ becomes $F_{\usigma^*} \left( T_n \right)$, and by Corollary~\ref{cor:first_moment} we have that $\E \left[ F_{\usigma^*} \left( T_n \right) \right] \simless n^{2 w \left( \utau \right) - 2}$. 
If $\usigma^* \in \tilde{\cU}^* \left( \utau \right) \cap \cD_0^*$, then, as mentioned above, $L_0 \left( \usigma^* \right) = \left\{ v^* \right\}$, and we have
\[
 \P \left( \argmin_{v \in \sigma^*} \dist_{T_n} \left( \psi \left( v \right), u_n \right) = v^* \, \middle| \, \cF_n \right) = \frac{f_{\psi \left( v^* \right)} \left( T_n \right)}{n}.
\]
So by summing over $\upsi^* : \usigma^* \to T_n$, if $\usigma^* \in \tilde{\cU}^* \left( \utau \right) \cap \cD_0^*$, then 
\[
 \E \left[ \sum_{\upsi^* : \usigma^* \to T_n}   \mathbf{1}_{\left\{ \usigma^* \in \cD_0^*, \argmin_{v \in \sigma^*} \dist_{T_n} \left( \psi \left( v \right), u_n \right) \in L_0 \left( \usigma^* \right) \right\}} \, \middle| \, \mathcal{F}_n \right] = \frac{1}{n} F_{\usigma} \left( T_n \right).
\]
Since $\usigma \in \cD \setminus \cD_0^*$ and $w\left( \usigma \right) \leq 2 w \left( \utau \right) - 1$, by Corollary~\ref{cor:first_moment} we have that $\E \left[ F_{\usigma} \left( T_n \right) \right] \simless n^{2 w \left( \utau \right) - 1}$ and thus $\E \left[ n^{-1} F_{\usigma} \left( T_n \right) \right] \simless n^{2 w \left( \utau \right) - 2}$. 
Putting everything together we thus obtain that 
\[
\E \left[ \left| \tilde{\cE}_{\utau}^{1,B} \left( T_{n+1} \right) \right| \right] \simless n^{2 w\left( \utau \right) - 2}. 
\]

\medskip

To estimate $\left| \tilde{\cE}_{\utau}^2 \left( T_{n+1} \right) \right|$ we can do the same thing as in part (a), and we obtain the same bound as for $\left| \tilde{\cE}_{\utau}^1 \left( T_{n+1} \right)\right|$ up to polylogarithmic factors in $n$. 
We omit the details. 
This concludes the proof of part (b).
\end{proof}

\subsection{Constructing the martingales}\label{sec:mg}

We now construct the martingales of Proposition~\ref{prop:mg} with the help of the recurrence relation of Lemma~\ref{lem:recurrence}. 
In order to show that these martingales are bounded in $L^2$, we use the moment estimates of Section~\ref{sec:moment}. 

\medskip

\begin{proof}\textbf{of Proposition~\ref{prop:mg}} 
 Fix a seed tree $S$ with $\left| S\right| = n_0 \geq 2$. For a decorated tree $\utau \in \cD_+$ and $n \geq 2$, define
\[
 \beta_n \left(\utau\right) := \prod_{j=2}^{n-1} \left( 1 + \frac{w\left( \utau \right)}{j} \right)^{-1}, \text{ \ \ when } \left| \utau \right| \geq 2; \text{ \ \ and \ \ } \beta_n \left(\utau\right) := \left( n \times \left[ n \right]_{w \left( \utau \right)} \right)^{-1}, \text{ \ \ when } \left| \utau \right| = 1.
\]
Note that when $\left| \utau \right| \geq 2$, we have $n^{-w\left( \utau \right)} \simless \beta_n \left(\utau\right) \simless n^{-w \left( \utau \right)}$. 

\medskip

We now construct, by induction on the order $\preccurlyeq$ on decorated trees, coefficients 
\[
 \left\{ a_n \left( \utau, \utau' \right) : \utau, \utau' \in \cD_+, \utau' \prec \utau, n \geq n_0 \right\}
\]
 such that
\begin{equation}\label{eq:a_n}
 a_n \left( \utau, \utau' \right) \simless 1, \qquad \qquad \qquad \Delta_n a \left( \utau, \utau' \right) \simless 1 / n,
\end{equation}
and 
\begin{equation}\label{eq:mg_def}
 M_{\utau}^{\left( S \right)} \left( n \right) = \beta_n \left(\utau\right) \left( F_{\utau} \left( \UA \left( n, S \right) \right) - \sum_{\utau' \in \cD_+: \utau' \prec \utau} a_n \left( \utau, \utau' \right) F_{\utau'} \left( \UA \left( n, S \right) \right) \right) 
\end{equation}
is a martingale. Importantly, we shall see that the coefficients $a_n \left( \utau, \utau' \right)$ do not depend on $S$. 
To simplify notation, in the following we omit dependence on $S$ and write $M_{\utau} \left( n \right)$ for $M_{\utau}^{\left( S \right)} \left( n \right)$. Also, as before, we write $T_n$ for $\UA \left( n, S \right)$. 

\medskip

First, when $\left| \utau \right| = 1$, we have $M_{\utau} \left( n \right) = \beta_n \left(\utau\right) F_{\utau} \left( T_n \right) = 1$, which is a martingale. 
Now fix $\utau \in \cD_+$ with $\left| \utau \right| \geq 2$. 
Assume that the coefficients $a_n \left( \usigma, \usigma' \right)$ have been constructed for every $\usigma, \usigma' \in \cD_+$ such that $\usigma' \prec \usigma \prec \utau$ and every $n \geq n_0$, and that they have the desired properties. 
We first claim that there exist constants $\left\{ b_n \left( \usigma, \usigma' \right) : \usigma' \prec \usigma \prec \utau, n \geq n_0 \right\}$ such that $b_n \left( \usigma, \usigma' \right) \simless 1$ and 
\begin{equation}\label{eq:inv}
 F_{\usigma} \left( T_n \right) = \frac{1}{\beta_n \left(\usigma\right)} M_{\usigma} \left( n \right) + \sum_{\usigma' \in \cD_+: \usigma' \prec \usigma} \frac{b_n \left( \usigma, \usigma' \right)}{\beta_n \left( \usigma' \right)} M_{\usigma'} \left( n \right)
\end{equation}
for $n \geq n_0$. To see this, define the matrix $A_n = \left( A_n \left( \usigma, \usigma' \right) \right)_{\usigma, \usigma' \prec \utau}$ by $A_n \left( \usigma, \usigma' \right) = -a_n \left( \usigma, \usigma' \right)$ if $\usigma' \prec \usigma$, $A_n \left( \usigma, \usigma' \right) = 1$ if $\usigma = \usigma'$, and $A_n \left( \usigma, \usigma' \right) = 0$ otherwise. 
Then, using~\eqref{eq:mg_def}, we have for every $n \geq n_0$ the following equality of vectors indexed by $\usigma \in \cD_+$ such that $\usigma \prec \utau$:
\begin{equation}\label{eq:mx_vec}
 \left( \frac{1}{\beta_n \left( \usigma \right)} M_{\sigma} \left( n \right) \right)_{\usigma \prec \utau} = A_n \cdot \left( F_{\usigma} \left( T_n \right) \right)_{\usigma \prec \utau}.
\end{equation}
We can write $\left\{ \usigma \in \cD_+ : \usigma \prec \utau \right\} = \left\{ \usigma_1, \dots, \usigma_K \right\}$ in such a way that $\usigma_i \prec \usigma_j$ implies $i < j$. 
With this convention, $A_n$ is a lower triangular matrix with all diagonal entries equal to $1$ and all entries satisfying $A_n \left( \usigma, \usigma' \right) \simless 1$. 
Therefore $A_n$ is invertible, and its inverse also satisfies these properties. 
That is, if we write $A_n^{-1} = \left( b_n \left( \usigma, \usigma' \right) \right)_{\usigma, \usigma' \prec \utau}$, then $A_n^{-1}$ is a lower triangular matrix that satisfies $b_n \left( \usigma, \usigma \right) = 1$ for all $\usigma \prec \utau$, and $b_n \left( \usigma, \usigma' \right) \simless 1$ for all $\usigma, \usigma' \prec \utau$. 
So~\eqref{eq:inv} follows directly from~\eqref{eq:mx_vec}. 

\medskip

Note that we can write equation~\eqref{eq:rec} of Lemma~\ref{lem:recurrence} more compactly as follows:
\begin{equation}\label{eq:rec2}
 \E \left[ F_{\utau} \left( \UA \left( n + 1, S \right) \right) \, \middle| \, \cF_n \right] = \left( 1 + \frac{w \left( \utau \right)}{n} \right) F_{\utau} \left( \UA \left( n, S \right) \right) + \frac{1}{n} \sum_{\utau' \in \cD: \utau' \prec \utau} c \left( \utau, \utau' \right) F_{\utau'} \left( \UA \left( n, S \right) \right),
\end{equation}
for appropriately defined constants $\left\{ c \left( \utau, \utau' \right) : \utau, \utau' \in \cD, \utau' \prec \utau \right\}$,
and note that since $\utau \in \cD_+$, we have $c \left( \utau, \utau' \right) = 0$ if $\utau' \notin \cD_+$. 
Therefore, using~\eqref{eq:rec2} and~\eqref{eq:inv}, together with the identities $\beta_{n+1} \left( \utau \right) \left( 1 + w \left( \utau \right) / n \right) = \beta_n \left( \utau \right)$ and $\beta_{n+1} \left( \utau \right) n^{-1} = \beta_n \left( \utau \right) \left( n + w \left( \utau \right) \right)^{-1}$, we have for $n \geq n_0$ that 
\begin{multline*}
 \E \left[ \beta_{n+1} \left( \utau \right) F_{\utau} \left( T_{n+1} \right) \, \middle| \, \cF_n \right] = \beta_n \left( \utau \right) F_{\utau} \left( T_n \right) + \frac{\beta_n \left( \utau \right)}{n+w\left( \utau \right)} \sum_{\utau' \in \cD_+: \utau' \prec \utau} c \left( \utau, \utau' \right) F_{\utau'} \left( T_n \right) \\
 = \beta_n \left( \utau \right) F_{\utau} \left( T_n \right) + \frac{1}{n+w\left( \utau \right)} \sum_{\usigma \in \cD_+ : \usigma \prec \utau} \left( c \left( \utau, \usigma \right) + \sum_{\utau' \in \cD_+ : \usigma \prec \utau' \prec \utau} c \left( \utau, \utau' \right) b_n \left( \utau', \usigma \right) \right) \frac{\beta_n \left( \utau \right)}{\beta_n \left( \usigma \right)} M_{\usigma} \left( n \right).
\end{multline*}
For $n \geq n_0$ define
\[
 \overline{a}_n \left( \utau, \usigma \right) = \sum_{j=n_0}^{n-1} \frac{1}{j+ w\left( \utau \right)} \left( c \left( \utau, \usigma \right) + \sum_{\utau' \in \cD_+ : \usigma \prec \utau' \prec \utau} c \left( \utau, \utau' \right) b_j \left( \utau', \usigma \right) \right) \frac{\beta_j \left( \utau \right)}{\beta_j \left( \usigma \right)}.
\]
We thus have
\[
 \E \left[ \beta_{n+1} \left( \utau \right) F_{\utau} \left( T_{n+1} \right) \, \middle| \, \cF_n \right] = \beta_n \left( \utau \right) F_{\utau} \left( T_n \right) + \sum_{\usigma \in \cD_+ : \usigma \prec \utau} \left( \overline{a}_{n+1} \left( \utau, \usigma \right) - \overline{a}_n \left( \utau, \usigma \right) \right) M_{\usigma} \left( n \right).
\]
By our induction hypothesis, $\left\{M_{\usigma} \left( n \right) \right\}_{n \geq n_0}$ is an $\left( \cF_n \right)$-martingale for every $\usigma \prec \utau$, and consequently $\beta_n \left( \utau \right) F_{\utau} \left( T_n \right) - \sum_{\usigma \in \cD_+ : \usigma \prec \utau}  \overline{a}_n \left( \utau, \usigma \right) M_{\usigma} \left( n \right)$ is also an $\left( \cF_n \right)$-martingale. 
By~\eqref{eq:mg_def} we have
\begin{multline*}
 \beta_n \left( \utau \right) F_{\utau} \left( T_n \right) - \sum_{\usigma \in \cD_+ : \usigma \prec \utau}  \overline{a}_n \left( \utau, \usigma \right) M_{\usigma} \left( n \right) \\
\begin{aligned}
 &=  \beta_n \left( \utau \right) F_{\utau} \left( T_n \right) - \sum_{\usigma \in \cD_+ : \usigma \prec \utau}  \overline{a}_n \left( \utau, \usigma \right) \beta_n \left( \usigma \right) \left( F_{\usigma} \left( T_n \right) - \sum_{\utau' \in \cD_+ : \utau' \prec \usigma} a_n \left( \usigma, \utau' \right) F_{\utau'} \left( T_n \right) \right) \\
 &= \beta_n \left( \utau \right) \left( F_{\utau} \left( T_n \right) - \sum_{\usigma \in \cD_+ : \usigma \prec \utau} \left[ \overline{a}_n \left( \utau, \usigma \right) \frac{\beta_n \left( \usigma \right)}{\beta_n \left( \utau \right)}  - \sum_{\utau' \in \cD_+ : \usigma \prec \utau' \prec \utau} \overline{a}_n \left( \utau, \utau' \right) a_n \left( \utau', \usigma \right) \frac{\beta_n \left( \utau' \right)}{\beta_n \left( \utau \right)} \right] F_{\usigma} \left( T_n \right) \right).
\end{aligned}
\end{multline*}
So if we set 
\[
 a_n \left( \utau, \usigma \right) := \overline{a}_n \left( \utau, \usigma \right) \frac{\beta_n \left( \usigma \right)}{\beta_n \left( \utau \right)}  - \sum_{\utau' \in \cD_+ : \usigma \prec \utau' \prec \utau} \overline{a}_n \left( \utau, \utau' \right) a_n \left( \utau', \usigma \right) \frac{\beta_n \left( \utau' \right)}{\beta_n \left( \utau \right)},
\]
then it is clear that $\left\{ M_{\utau} \left( n \right) \right\}_{n \geq n_0}$ defined as in~\eqref{eq:mg_def} is a martingale. 

\medskip

Now let us establish that the coefficients are of the correct order, i.e., let us show~\eqref{eq:a_n}. First note that $\left( n + w \left( \utau \right) \right)^{-1} \simless 1/n$, and that when $\left| \utau \right| \geq 2$, $\beta_n \left( \utau \right) n^{w\left( \utau \right)}$ has a positive and finite limit as $n \to \infty$. Therefore a simple computation shows that for $\usigma, \usigma' \in \cD_+$ with $\left| \usigma \right|, \left| \usigma' \right| \geq 2$, we have
\[
 \frac{\beta_n \left( \usigma \right)}{\beta_n \left( \usigma' \right)} \simless n^{w \left( \usigma' \right) - w \left( \usigma \right)} \qquad \text{ and } \qquad \Delta_n \frac{\beta \left( \usigma \right)}{\beta \left( \usigma' \right)} \simless n^{w \left( \usigma' \right) - w \left( \usigma \right) - 1}.
\]
Furthermore, by the induction hypothesis we have that $b_n \left( \usigma, \usigma' \right) \simless 1$ for every $\usigma, \usigma' \prec \utau$. 
From the definition of $\overline{a}_n \left( \utau, \usigma \right)$ we then immediately get that $\Delta_n \overline{a} \left( \utau, \usigma \right) \simless n^{w \left( \usigma \right) - w \left( \utau \right) - 1}$, and consequently also $\overline{a}_n \left( \utau, \usigma \right) \simless n^{w \left( \usigma \right) - w \left( \utau \right)}$, for every $\usigma \in \cD_+$ such that $\usigma \prec \utau$ and $\left| \usigma \right| \geq 2$. 
So for every $\usigma \in \cD_+$ such that $\usigma \prec \utau$ and $\left| \usigma \right| \geq 2$ we have that 
\begin{equation}\label{eq:coeff_check}
 \overline{a}_n \left( \utau, \usigma \right) \frac{\beta_n \left( \usigma \right)}{\beta_n \left( \utau \right)} \simless 1 \qquad \text{ and } \qquad \Delta_n \left( \overline{a} \left( \utau, \usigma \right) \frac{\beta \left( \usigma \right)}{\beta \left( \utau \right)} \right) \simless \frac{1}{n}.
\end{equation}
One can easily check that~\eqref{eq:coeff_check} holds also when $\left| \usigma \right| = 1$. Now combining all of these estimates with the definition of $a_n \left( \utau, \usigma \right)$, we get that~\eqref{eq:a_n} holds. 
This completes the induction.

\medskip

Finally, what remains to show is that the martingales $M_{\utau}$ are bounded in $L^2$. Since $M_{\utau}$ is a martingale, its increments are orthogonal in $L^2$, and so
\[
 \E \left[ M_{\utau} \left( n \right)^2 \right] = \sum_{j=n_0}^{n-1} \E \left[ \left( M_{\utau} \left( j+1 \right) - M_{\utau} \left( j \right) \right)^2  \right] + \E \left[ M_{\utau} \left( n_0 \right)^2 \right].
\]
Clearly $\E \left[ M_{\utau} \left( n_0 \right)^2 \right] < \infty$ and so it suffices to show that 
\[
 \sum_{n=n_0}^{\infty} \E \left[ \left( M_{\utau} \left( n + 1 \right) - M_{\utau} \left( n \right) \right)^2 \right] < \infty.
\]
Recalling the definition of $M_{\utau}$ from~\eqref{eq:mg_def} we have
\[
 \E \left[ \left( \Delta_n \left( M_{\utau} \right) \right)^2 \right] = \E \left[ \left( \Delta_n \left( \beta_{\cdot} \left( \utau \right) F_{\utau} \left( T_{\cdot} \right) \right) - \sum_{\utau \in \cD_+ : \utau' \prec \utau} \Delta_n \left( \beta_{\cdot} \left( \utau \right) a_{\cdot} \left( \utau, \utau' \right) F_{\utau'} \left( T_{\cdot} \right) \right) \right)^2 \right],
\]
where the dots in the subscripts denote dependence on $n$, on which the difference operator $\Delta_n$ acts. 
By the Cauchy-Schwarz inequality, there exists a positive and finite constant $c$ that depends only on $\utau$ such that for every $n \geq n_0$, the quantity $c \times \E \left[ \left( \Delta_n \left( M_{\utau} \right) \right)^2 \right]$ is bounded from above by
\begin{equation}\label{eq:mg2_ub}
 \E \left[ \left( \Delta_n \left( \beta_{\cdot} \left( \utau \right) F_{\utau} \left( T_{\cdot} \right) \right) \right)^2 \right] + \sum_{\utau \in \cD_+ : \utau' \prec \utau} \E \left[ \left( \Delta_n \left( \beta_{\cdot} \left( \utau \right) a_{\cdot} \left( \utau, \utau' \right) F_{\utau'} \left( T_{\cdot} \right) \right) \right)^2 \right] 
\end{equation}
Since $\Delta_n \left( \beta_{\cdot} \left( \utau \right) F_{\utau} \left( T_{\cdot} \right) \right) = \beta_{n+1} \left( \utau \right) \Delta_n \left( F_{\utau} \left( T_{\cdot} \right) \right) + \left( \Delta_n \left( \beta_{\cdot} \left( \utau \right) \right) \right) F_{\utau} \left( T_n \right)$, we have that
\[
 \E \left[ \left( \Delta_n \left( \beta_{\cdot} \left( \utau \right) F_{\utau} \left( T_{\cdot} \right) \right) \right)^2 \right] \leq 2 \left( \beta_{n+1} \left( \utau \right) \right)^2 \E \left[ \left( \Delta_n \left( F_{\utau} \left( T_{\cdot} \right) \right) \right)^2 \right] + 2 \left( \Delta_n \left( \beta_{\cdot} \left( \utau \right) \right) \right)^2 \E \left[ F_{\utau} \left( T_n \right)^2 \right].
\]
We have seen that $\left( \beta_{n+1} \left( \utau \right) \right)^2 \simless n^{-2w\left( \utau \right)}$ and $\left( \Delta_n \left( \beta_{\cdot} \left( \utau \right) \right) \right)^2 \simless n^{-2w\left( \utau \right) - 2}$, and by Lemma~\ref{lem:second_moment} we have that $\E \left[ F_{\utau} \left( T_n \right)^2 \right] \simless n^{2w\left( \utau \right)}$ and $\E \left[ \left( \Delta_n \left( F_{\utau} \left( T_{\cdot} \right) \right) \right)^2 \right] \simless n^{2 w\left( \utau \right) - 2}$. 
Putting these together we thus have that $\E \left[ \left( \Delta_n \left( \beta_{\cdot} \left( \utau \right) F_{\utau} \left( T_{\cdot} \right) \right) \right)^2 \right] \simless n^{-2}$. 
For the other terms in~\eqref{eq:mg2_ub} we similarly have
\begin{multline*}
 \E \left[ \left( \Delta_n \left( \beta_{\cdot} \left( \utau \right) a_{\cdot} \left( \utau, \utau' \right) F_{\utau'} \left( T_{\cdot} \right) \right) \right)^2 \right] \\
\leq 2 \left( a_{n+1} \left( \utau, \utau' \right) \right)^2 \E \left[ \left( \Delta_n \left( \beta_{\cdot} \left( \utau \right)  F_{\utau'} \left( T_{\cdot} \right) \right) \right)^2 \right] 
+ 2 \left( \Delta_n \left( a_{\cdot} \left( \utau, \utau' \right) \right) \right)^2 \E \left[ \left( \beta_n \left( \utau \right) F_{\utau'} \left( T_n \right) \right)^2 \right].
\end{multline*}
We have seen that $\left( a_{n+1} \left( \utau, \utau' \right) \right)^2 \simless 1$ and $\left( \Delta_n \left( a_{\cdot} \left( \utau, \utau' \right) \right) \right)^2 \simless n^{-2}$. Furthermore, by Lemma~\ref{lem:second_moment} we have that $\E \left[ \left( \beta_n \left( \utau \right) F_{\utau'} \left( T_n \right) \right)^2 \right] \simless n^{2w \left( \utau' \right) - 2 w \left( \utau \right)} \leq 1$, and similarly to the computation above we have that 
$\E \left[ \left( \Delta_n \left( \beta_{\cdot} \left( \utau \right)  F_{\utau'} \left( T_{\cdot} \right) \right) \right)^2 \right] \simless n^{2w \left( \utau' \right) - 2 w \left( \utau \right) - 2} \leq n^{-2}$. 
Putting everything together we get that
\[
 \E \left[ \left( M_{\utau} \left( n + 1 \right) - M_{\utau} \left( n \right) \right)^2 \right] \simless n^{-2},
\]
which is summable, so $M_{\tau}$ is indeed bounded in $L^2$.
\end{proof}

%%%%%%%%%%%%%%%%%%%%%%%%%%%%%%%%%%%%%%%%%%%%%%%
\section{Discussion} \label{sec:discussion} %%%
%%%%%%%%%%%%%%%%%%%%%%%%%%%%%%%%%%%%%%%%%%%%%%%

We conclude with a comparison of our proof to that of~\cite{curien2014scaling}, and with a list of open problems.

\subsection{Comparison to~\cite{curien2014scaling}}\label{sec:curien}

As discussed in Section~\ref{sec:related}, the key difference in our proof for uniform attachment compared to the proof of~\cite{curien2014scaling} for preferential attachment is the underlying family of statistics. 
For preferential attachment these are based on the degrees of the nodes, whereas for uniform attachment they are based on partition sizes when embedding a given tree, i.e., they are based on global balancedness properties of the tree. 

\medskip

The statistics $F_{\utau} \left( T \right)$ are defined in this specific way in order to make the analysis simpler. 
In particular, it is useful that $F_{\utau} \left( T \right)$ has a combinatorial interpretation as the number of decorated embeddings of $\utau$ in $T$, similarly to the statistics of~\cite{curien2014scaling}. 
However, the notion of a decorated embedding is different in the two settings. 
In~\cite{curien2014scaling}, arrows associated with the decorated tree $\utau$ are mapped by $\uphi$ to corners around the vertices of $\phi \left( \tau \right)$, or in other words, the decorations are local. 
In contrast, in the notion of a decorated embedding as defined in this paper, arrows associated with a decorated tree $\utau$ can be mapped to any vertex in the graph $T$, or in other words, the decorations are nonlocal/global. 

\medskip

While the general structure of our proof is identical to that of~\cite{curien2014scaling}, this local vs.\ global difference in the underlying statistics manifests itself in the details. 
In particular, the main challenge is the second moment estimate provided in Lemma~\ref{lem:second_moment}. 
Here, we associate to each decorated map $\uphi = \uphi_1 \times \uphi_2$ a decorated tree $\usigma$ and a decorated embedding $\upsi$ of it in $\UA \left( n, S \right)$. 
In the case of preferential attachment, the decorated tree $\usigma$ necessarily has all labels be positive, due to the decorations being local. 
However, in the case of uniform attachment, it might happen that a vertex of $\usigma$ has a label being zero, due to the global nature of decorations. 
This is the reason why we need to deal with decorated trees having labels being zero, in contrast with the preferential attachment model, where it suffices to consider decorated trees with positive labels. 
The recurrence relation and the subsequent moment estimates show that there is a subtlety in dealing with decorated trees having labels being zero, as it matters whether the vertices with label zero are leaves or not. 

\medskip

Finally, in our proof of the second moment estimate we also use the fact that the diameter of $\UA \left( n, S \right)$ is on the order of $\log \left( n \right)$ with high probability (see Lemma~\ref{lem:diam}). 
This is again due to the global nature of decorations, and such an estimate is not necessary in the case of preferential attachment.

\subsection{Open problems}\label{sec:open}

\cite{bubeck2014influencePAseed} list several open problems for the preferential attachment model, and these questions can also be asked for the uniform attachment model. We list here a few of them.

\begin{enumerate}
 \item What can be said about general uniform attachment graphs, where multiple edges are added at each step?

 \item Under what conditions on two tree sequences $\left( T_k \right)$, $\left( R_k \right)$ do we have $\lim_{k \to \infty} \delta \left( T_k, R_k \right) = 1$?

 \item Is it possible to give a combinatorial description of the metric $\delta$?

 \item A simple model that interpolates between uniform attachment and (linear) preferential attachment is to consider probabilities of connection proportional to the degree of the vertex raised to some power $\alpha$. 
The results of this paper show that for $\alpha = 0$ different seeds are distinguishable, while for $\alpha = 1$ this is shown by~\cite{bubeck2014influencePAseed} and~\cite{curien2014scaling}. 
What about for $\alpha \in \left( 0, 1 \right)$? Is $\delta_\alpha \left( S, T \right) > 0$ whenever $S$ and $T$ are non-isomorphic and have at least three vertices? What can be said about $\delta_\alpha \left( S, T \right)$ as a function of $\alpha$? Is it monotone in $\alpha$? Is it convex?
\end{enumerate}

%%%%%%%%%%%%%%%%%%%%%%%%
%%% Acknowledgements %%%
%%%%%%%%%%%%%%%%%%%%%%%%

\section*{Acknowledgements}

This work was done while R.E.\ and E.M.\ were visiting researchers and M.Z.R. was an intern at the Theory Group of Microsoft Research. They thank the Theory Group for their kind hospitality. S.B.\ thanks Luc Devroye for enlightening discussions on the uniform random recursive tree.
E.M.\ was supported by
NSF grants DMS 1106999 and CCF 1320105, 
by DOD ONR grant N00014-14-1-0823, and 
by grant 328025 from the Simons Foundation.

%%%%%%%%%%%%%%%%%%
%%% References %%%
%%%%%%%%%%%%%%%%%%

\bibliographystyle{plainnat}
\bibliography{bib}

%%%%%%%%%%%%%%%%
%%% Appendix %%%
%%%%%%%%%%%%%%%%

% \newpage

\appendix

\section{Facts about the beta-binomial distribution}\label{app:beta}

We prove here Facts~\ref{factbbmoments} and~\ref{factbbsmall} stated in Section~\ref{sec:prelim}. 
Let $M_k$, $k=\alpha + \beta, \alpha + \beta + 1, ...$, be the martingale associated with the standard P\'olya urn process with starting state $(\alpha, \beta)$. 
In other words, the martingale $M_k$ is defined by $M_{\alpha + \beta} = \frac{\alpha}{\alpha + \beta}$ and
$$(k+1) M_{k+1} = \begin{cases}
k M_k + 1 & \mbox{ with probability } M_k \\
k M_k & \mbox{ with probability } 1 - M_k
\end{cases}
$$
independently for different values of $k$. 
Note that for $n \geq \alpha + \beta$,  $n M_n \stackrel{d}{=} B_{\alpha, \beta, n - \alpha - \beta}$, so all results for the martingale $M_n$ transfer to results for $B_{\alpha, \beta, n - \alpha - \beta}$.  
Define $M_\infty = \lim_{k \to \infty} M_k$, and note that this limit exists almost surely by the martingale convergence theorem. 
It is a well-known fact about P\'olya urns that $M_{\infty}$ has a beta distribution with parameters $\alpha$ and $\beta$, i.e., the density of $M_{\infty}$ with respect to the Lebesgue measure is
\[
 h\left( x \right) = \frac{\Gamma(\alpha + \beta)}{\Gamma(\alpha) \Gamma(\beta)} x^{\alpha - 1} (1-x)^{\beta - 1} \mathbf{1}_{\left\{ x \in \left[0,1 \right] \right\}}.
\]
By the formula for the moments of $M_\infty$ (see, e.g., \cite[Chapter 21]{CUD}), we have
$$
\EE[M_\infty^p] = \prod_{j=0}^{p-1} \frac{\alpha + j}{ \alpha + \beta + j} \leq \left (\frac{\alpha + p}{\alpha + \beta} \right )^p \leq (p+1)^p \left (\frac{\alpha}{\alpha + \beta} \right )^p, ~~ \forall p \in \mathbb{N}.
$$
Moreover, since $M_n$ is a martingale, $M_n^p$ is a submartingale for all $p \geq 1$, and thus $\E \left[ M_n^p \right] \leq \E \left[ M_{\infty}^p \right]$. So we have that
$$
\EE[ (n M_n)^p ] \leq n^p (p+1)^p \left (\frac{\alpha}{\alpha + \beta} \right )^p,
$$
which establishes Fact~\ref{factbbmoments} with $C\left( p \right) = p + 1$.

\medskip

Next, in order to prove Fact~\ref{factbbsmall}, we first use the formula for the negative first moment of $M_\infty$ (see, e.g., \cite[Chapter 21]{CUD}): 
for every $\alpha > 1$ we have $\E \left[ M_{\infty}^{-1} \right] = \left( \alpha + \beta - 1 \right) / \left( \alpha - 1 \right)$. 
Thus by Markov's inequality we have that $\P \left( M_\infty < z \right) \leq z \left( \alpha + \beta - 1 \right) / \left( \alpha - 1 \right)$ for every $z > 0$, and thus
\begin{equation}\label{eq:factbbsmall_proof}
 \P \left( M_{\infty} < t \frac{\alpha}{\alpha+\beta} \right) \leq 2t.
\end{equation}
In the case that $\alpha = 1$, we have $h \left( x \right) \leq \beta$ which implies that $\int_0^{\frac{t}{\beta + 1}} h(x) dx \leq t$. 
We conclude that~\eqref{eq:factbbsmall_proof} holds for any $\alpha, \beta \geq 1$. 
Since $M_k$ is a positive martingale, we have 
\[
 \P \left( M_{\infty} \leq 2 z \, \middle| \, M_n \leq z \right) \geq 1/2, ~~ \forall z \in \left( 0, 1 \right).
\]
Combining this inequality with~\eqref{eq:factbbsmall_proof} gives
\[
 \P \left( M_n \leq t \frac{\alpha}{\alpha + \beta} \right) \leq 8t,
\]
and formula~\eqref{eq:fact2} then follows with $C = 8$.

\section{Estimates on sequences}\label{app:est}

\begin{lemma}\label{lem:est}
 Suppose that $\left\{ a_n \right\}_{n \geq 1}$ is a sequence of nonnegative real numbers and that there exists $n_0$ such that $a_{n_0} > 0$. Let $\alpha$ be a positive integer. 
\begin{enumerate}[(a)]
 \item If there exists $N$ such that $a_{n+1} \geq \left( 1 + \alpha / n \right) a_n$ for every $n \geq N$, then $\liminf_{n \to \infty} a_n / n^{\alpha} > 0$. 
 \item If there exist constants $c$, $\gamma$, and $N$ such that for every $n \geq N$,
\[
 a_{n+1} \leq \left( 1 + \alpha / n \right) a_n + c \left( \log \left( n \right) \right)^{\gamma} n^{\alpha - 1},
\]
then $a_n \simless n^\alpha$.
 \item If there exist constants $c$, $\gamma$, and $N$ such that for every $n \geq N$,
\[
 a_{n+1} \leq \left( 1 + \alpha / n \right) a_n + c \left( \log \left( n \right) \right)^{\gamma} n^{\alpha},
\]
then $a_n \simless n^{\alpha+1}$.
\end{enumerate}
\end{lemma}
\begin{proof}
 (a) By the assumption we have that $a_n \geq a_{n_0} \exp \left( \sum_{j= n_0}^{n-1} \log \left( 1  + \alpha / j \right) \right)$, where $a_{n_0} > 0$. 
For $0 \leq x \leq 1$ we have that $\log \left( 1 + x \right) \geq x - x^2$, and so using the fact that $\sum_{j=1}^\infty 1/j^2 < \infty$, we have that there exists $c > 0$ such that $a_n \geq c \exp \left( \alpha \sum_{j= n_0 \vee \alpha}^{n-1} 1/j \right)$. To conclude, recall that $\sum_{j=1}^{n-1} 1/j > \log \left( n \right)$.

\medskip

(b) Let $b_n := a_n / n^{\alpha}$. We then have that $b_{n+1} \leq \left( 1 + \alpha / n \right) \left( n / \left( n + 1 \right) \right)^{\alpha} b_n + c \left( \log \left( n \right) \right)^{\gamma} / n$. 
There exists a constant $c''= c'' \left( \alpha \right)$ such that $\left( n / \left( n+ 1 \right) \right)^{\alpha} \leq 1 - \alpha / n + c'' / n^2$ for every $n \geq 1$. 
Therefore there exists a constant $c' = c' \left( \alpha \right)$ such that $\left( 1 + \alpha / n \right) \left( n / \left( n + 1 \right) \right)^{\alpha} \leq 1 + c' / n^2$ for every $n \geq 1$. 
Thus we have that $b_{n+1} \leq \left( 1 + c' / n^2 \right) b_n + c \left( \log \left( n \right) \right)^\gamma / n$, and iterating this we get that
\[
 b_n \leq b_1 \prod_{j=1}^{n-1} \left( 1 + c' / j^2 \right) + \sum_{j=1}^{n-1} \left(  \prod_{i=j+1}^{n-1} \left( 1 + c' / i^2 \right) \right) c \left( \log \left( j \right) \right)^\gamma / j.
\]
Since $\prod_{j=1}^{\infty} \left( 1 + c' / j^2 \right) < \infty$, we immediately get that $b_n \simless 1$, and so $a_n \simless n^{\alpha}$. 

\medskip

(c) This is similar to (b) so we do not repeat the argument.
\end{proof}

\section{Tail behavior of the diameter}\label{app:diam}

We reproduce a simple argument of~\cite{devroye2011long} to obtain a tail bound for the diameter $\diam \left( \UA \left(n, S \right) \right)$ of a uniform attachment tree. 
\begin{lemma}\label{lem:diam}
 For every seed tree $S$ there exists a constant $C = C\left( S \right)$ such that for every $K > 20$ we have
\[
\P \left( \diam \left( \UA \left( n, S \right) \right) > K \log \left( n \right) \right) \leq \frac{C\left( S \right)}{n^{K/2}}.
\]
\end{lemma}
\begin{proof}
 First, if we set $C \left( S \right) := \left( \P \left( \UA \left( \left| S \right|, S_2 \right) = S \right) \right)^{-1}$, then we have
\begin{align*}
 \P \left( \diam \left( \UA \left( n, S \right) \right) > K \log \left( n \right) \right) &= \P \left( \diam \left( \UA \left( n, S_2 \right) \right) > K \log \left( n \right) \, \middle| \, \UA \left( \left| S \right|, S_2  \right) = S \right) \\
&\leq C\left( S \right) \P \left( \diam \left( \UA \left( n, S_2 \right) \right) > K \log \left( n \right) \right),
\end{align*}
so it remains to bound the tail of $\diam \left( \UA \left( n, S_2 \right) \right)$. 

\medskip

For notational convenience, shift the names of the vertices so that they consist of the set $\left\{ 0, 1, \dots, n-1 \right\}$ (instead of $\left\{1, 2, \dots, n \right\}$), and call vertex $0$ the root. 
With this convention, the label of the parent of vertex $j$ is distributed as $\left\lfloor j U \right\rfloor$ where $U$ is uniform on $\left[ 0, 1 \right]$. 
Similarly, an ancestor $\ell$ generations back has a label distributed like $\left\lfloor \dots \left\lfloor \left\lfloor j U_1 \right\rfloor U_2 \right\rfloor \dots U_{\ell} \right\rfloor$, where the $U_i$'s are i.i.d.\ uniform on $\left[0,1 \right]$. 

\medskip

Define $R_j$ to be the distance from vertex $j$ to the root. By the triangle inequality we have that $\diam \left( \UA \left( n, S \right) \right) \leq 2 \max_{1 \leq j \leq n-1} R_j$, so it suffices to bound the tail of this latter quantity. Using a union bound, it then suffices to bound the tail of $R_j$ for each $j$. 
Now notice that $R_j \leq \min \left\{ t : j U_1 \dots U_t < 1 \right\}$. 
Consequently, for any $\lambda > 0$ we have
\[
 \P\left( R_j > t \right) \leq \P \left( j U_1 \dots U_t \geq 1 \right) \leq \E \left[ \left( j U_1 \dots U_t \right)^{\lambda} \right] = j^{\lambda} \left( \lambda + 1 \right)^{-t}.
\]
This is optimized by choosing $\lambda = t/\log \left( j \right) - 1$ (provided $t > \log \left( j \right)$) to obtain
\[
 \P\left( R_j > t \right) \leq \exp \left( t - \log \left( j \right) - t \log \left( \frac{t}{\log \left( j \right)} \right) \right) \leq \exp \left( t - t \log \left( \frac{t}{\log \left( n \right)} \right) \right),
\]
when $j\leq n$. Putting everything together we get that
\begin{align*}
 \P \left( \diam \left( \UA \left( n, S_2 \right) \right) > K \log \left( n \right) \right) &\leq \P \left( \max_{1 \leq j \leq n-1} R_j > \frac{K}{2} \log \left( n \right) \right) \leq \sum_{j=1}^{n-1} \P \left( R_j > \frac{K}{2} \log \left( n \right) \right) \\
&\leq n^{-\frac{K}{2} \log \left( \frac{K}{2} \right) + \frac{K}{2} + 1} \leq n^{-K/2},
\end{align*}
where the last inequality holds when $K > 20$.
\end{proof}

\end{document}